\newif\ifHAL
\HALtrue

\ifHAL
\documentclass[a4paper,10pt]{article}

\usepackage{a4wide,geometry}
\usepackage[normalem]{ulem}

\usepackage[utf8]{inputenc}
\usepackage[T1]{fontenc}
\usepackage[english]{babel}

\usepackage{amsmath,amsthm,amssymb}

\newtheorem{theorem}{Theorem}[section]
\newtheorem{proposition}[theorem]{Proposition}
\newtheorem{lemma}[theorem]{Lemma}
\newtheorem{corollary}[theorem]{Corollary}
\newtheorem{remark}[theorem]{Remark}

\numberwithin{equation}{section}
\else
\documentclass{mcom-l}

\usepackage{amssymb}

\usepackage[breaklinks,bookmarks=false]{hyperref}
\hypersetup{%
colorlinks,%
linkcolor=blue,%
citecolor=blue,%
urlcolor=blue,%
plainpages=false,%
pdfwindowui=false,%
pdfstartview={FitH}%
}

\newtheorem{theorem}{Theorem}[section]
\newtheorem{proposition}[theorem]{Proposition}
\newtheorem{lemma}[theorem]{Lemma}

\newtheorem{remark}[theorem]{Remark}

\numberwithin{equation}{section}
\fi

\usepackage{xcolor}

\input mydef.sty

\newcommand{\rev}[1]{#1} 

\begin{document}

\ifHAL
\title{A priori and a posteriori analysis of the discontinuous Galerkin approximation
of the time-harmonic Maxwell's equations under minimal regularity assumptions}

\author{T. Chaumont-Frelet\footnotemark[1], A. Ern\footnotemark[2]}

\footnotetext[1]{Inria Univ. Lille and Laboratoire Paul Painlev\'e, 59655 Villeneuve-d'Ascq, France}

\footnotetext[2]{CERMICS, Ecole nationale des ponts et chauss\'ees, IP Paris, 6 \& 8 avenue B.~Pascal, 77455 Marne-la-Vall\'{e}e, France and Inria, 48 Rue Barrault, 75647 Paris, France}

\date{Draft version, \today}
\fi

\maketitle

\begin{abstract} 
We derive a priori and a posteriori error estimates for the discontinuous Galerkin (dG)
approximation of the time-harmonic Maxwell's equations. Specifically, we consider an interior
penalty dG method, and establish error estimates that are valid under minimal regularity
assumptions and involving constants that do not depend on the frequency for sufficiently
fine meshes. The key result of our a priori error analysis is that the dG solution is
asymptotically optimal in an augmented energy norm that contains the dG stabilization.
Specifically, up to a constant that tends to one as the mesh is refined,
the dG solution is as accurate as the best approximation in the same
norm. \rev{The main insight is that the quantities controlling the smallness of the 
mesh size are essentially those already appearing in the conforming setting.}
We also show that for fine meshes, the inf-sup stability constant
is as good as the continuous one up to a factor two. Concerning the a posteriori
analysis, we consider a residual-based error estimator under the assumption of
piecewise constant material properties. We derive a global upper bound and local lower
bounds on the error with constants that \textup{(i)} only depend on the shape-regularity
of the mesh if it is sufficiently refined and \textup{(ii)} are independent of the
stabilization bilinear form.
\end{abstract}

\ifHAL

\fi

\section{Introduction}

Let $\Dom\subset \Real^d$, $d=3$, be an open, bounded, Lipschitz polyhedron 
with boundary $\front$ and outward unit normal $\bn_\Dom$. 
We do not make any simplifying assumption on the topology of $\Dom$.
We use boldface fonts for vectors, vector fields, and
functional spaces composed of such fields. More details
on the notation are given in Section~\ref{sec:continuous}. 

Given a positive real number $\omega>0$ representing a frequency and a source
term $\bJ: \Dom \to \mathbb R^3$, and focusing for simplicity on homogeneous
Dirichlet boundary conditions (a.k.a.~perfect electric conductor boundary conditions),
the model problem consists in finding $\bE: \Dom \to \mathbb R^3$ such that
\begin{subequations}
\label{eq_maxwell_strong}
\begin{alignat}{2}
-\omega^2 \eps \bE+\ROT(\bmu^{-1}\ROT \bE) &= \bJ &\quad&\text{in $\Dom$},
\\
\bE \CROSS \bn_\Dom &= \bzero &\quad&\text{on $\partial \Dom$},
\end{alignat}
\end{subequations}
where $\eps$ represents the electric permittivity of the
materials contained in $\Dom$ and $\bmu$ their magnetic permeability. 
Both material properties can vary in $\Dom$ and take symmetric
positive-definite values with eigenvalues uniformly bounded from above and
from below away from zero. 
We assume that $\omega$ is not a resonant frequency,
so that \eqref{eq_maxwell_strong} is uniquely solvable in $\Hrotz$
for every $\bJ$ in the topological dual space $\Hrotz'$. 
The time-harmonic
Maxwell's equations \eqref{eq_maxwell_strong} are one of the central models
of electrodynamics. Therefore, efficient discretizations are
a cornerstone for the computational modelling of electromagnetic 
wave propagation \cite{Hiptmair_Acta_Num_2002,Monk_book_2003}. In this work,
we focus on the discontinuous Galerkin (dG) method.

DG methods employ approximation spaces composed of nonconforming (discontinuous, broken) 
polynomials on the mesh. They are attractive since they easily allow for more
flexibility in the mesh and for local variations of the polynomial degree. \rev{[--]} 
DG methods exist in many flavors. One popular approach for the Poisson model
problem is the interior penalty dG method, which hinges on a consistency term
involving the mean-value of the normal flux at the mesh faces, possibly a symmetry term,
and a stabilization term penalizing the jumps across the mesh interfaces and the values
at the mesh boundary faces (see, e.g., \cite{ArBCM:01,DiPEr:12} and the references therein).
The expression of the consistency term involving the mean-value of the normal
flux is convenient for efficient implementation, but for the analysis, it is useful
to consider an (equivalent) reformulation involving jump liftings. The approach was
first considered in \cite{BasRe:97} and analyzed in \cite{BMMPR:00}. One important outcome is
the notion of discrete gradient obtained by adding the jump liftings to the broken (piecewise)
gradient. Indeed, the discrete gradient enjoys a compactness property that plays a central role
in various nonlinear problems \cite{BurEr:08,BufOr:09,DiPEr:10}. Another attractive feature
is that the discrete gradient admits a bounded extension to $H^1$, whereas 
the standard consistency term can only be extended to $H^{1+s}$ with $s>\frac12$. 
It is also possible to define bounded extensions
of the consistency term to $H^{1+s}$ for $s>0$ arbitrarily small by proceeding as in
\cite{ErnGu:22}.

In the context of the time-harmonic Maxwell's equations, interior penalty
dG methods were devised and analyzed in \cite{PeScM:02,HoPSS:05}, and the notion
of discrete curl, obtained by adding the liftings of the tangential jumps 
to the broken curl, has been considered in the method formulation 
and analysis. However, the combined use of discrete curls (allowing for minimal
regularity requirements) with Schatz's duality argument seems to be lacking in 
the literature, contrary to the case of the Helmholtz equation where such a result 
has been recently achieved in~\cite{Chaumont:23_ipdghelmholtz}. Our \rev{first} main contribution
is to fill this gap. Indeed, we show that an \emph{asymptotically optimal} error 
estimate holds true with an augmented energy norm including a nonconformity measure.
Asymptotic optimality means that the ratio between the approximation error
and the best-approximation error tends to one as the mesh size $h$ is sent to zero.
\rev{An important aspect in our analysis is that the (frequency dependent) quantities controlling the smallness of the 
mesh size are essentially those already appearing in the conforming setting.}
Our second main contribution is to establish asymptotic optimality
under minimal regularity assumptions\rev{. Specifically, we assume that 
the source term sits in $\Ldeuxd$} 
(rather than in the dual space $\Hrotz'$) with no further assumption on $\DIV\bJ$
\rev{and that the material coefficients are bounded from above and from below away from zero,
with no further regularity assumption on the exact solution.
The third main} contribution of the paper 
is an a posteriori error analysis \rev{in the present indefinite setting
and using, for the first time, a duality argument. We establish}
global upper bounds and local lower bounds
on the error, where the constants are independent of the frequency, again in the limit 
as the mesh is refined. \rev{An important insight is, once again, that the behavior of the
constants is the same as for a conforming approximation.}
The a posteriori error analysis is of residual-type and requires
to tighten slightly the assumption on the source term so that $\DIV \bJ \in \Ldeux$,
\rev{and we assume piecewise constant material properties.}

Let us put our results in perspective with the literature. Concerning 
the a priori error analysis, a quasi-optimal (but not asymptotically optimal)
error estimate under minimal regularity is derived in~\cite{LGSvdV:23}, but for
a different interior penalty dG method, where a Lagrange multiplier related to
the divergence constraint is introduced together with the corresponding stabilization
term. Moreover, asymptotically optimal error estimates for 
the time-harmonic Maxwell's equations approximated using conforming edge elements
have been derived quite recently in \cite{melenk_sauter_2023a} and
in \cite{ThCFE:23}. The analysis in \cite{melenk_sauter_2023a} considers impedance
boundary conditions (allowing for an explicit frequency analysis), but requires the domain 
boundary to be smooth and connected. Instead, the analysis in \cite{ThCFE:23}
considers Dirichlet boundary conditions
and allows for general domains, material coefficients, and right-hand side.
The present analysis leverages the ideas developed in \cite{ThCFE:23},
but needs to address two additional, nontrivial difficulties: (i) the lack of
strong consistency under minimal regularity, leading to the appearance of new terms
in the analysis related to the consistency defect; (ii) the nonconforming nature
of the approximation, which calls for a careful handling of the stabilization.
In particular, we notice that we allow for a rather general stabilization bilinear form
and provide explicit design assumptions for the analysis to hold true.

Concerning the a posteriori error analysis, we leverage the ideas
proposed in \cite{chaumontfrelet_vega_2022a} for the conforming edge
finite element approximation of the time-harmonic Maxwell's equations. 
Here, the novelty is twofold. First, we additionally deal with the consistency defect
and the presence of stabilization in the discontinuous Galerkin setting,
by extending ideas introduced in \cite{Chaumont:23_ipdghelmholtz} for the Helmholtz equation.
Moreover, we tighten some arguments from \cite{chaumontfrelet_vega_2022a}
in the proof of the error upper bound so that the involved constants only depend on the
shape-regularity parameter of the mesh, whereas in \cite{chaumontfrelet_vega_2022a}
some constants are frequency-dependent in the low-frequency regime. Specifically,
instead of invoking the regular decomposition results from
\cite[Theorem 2.1]{hiptmair_pechstein_2019} as in \cite{chaumontfrelet_vega_2022a},
we make use of Galerkin orthogonality on conforming test functions to invoke the regular
decomposition results from \cite[Theorem 1]{Schoberl:07}.
\rev{Finally, we observe that the general form of the error 
indicators is the same as the one derived in \cite{HoPeS:07} 
in the positive definite setting.}

The paper is organized as follows. In Section~\ref{sec:continuous}, we briefly
present the continuous setting, and in Section~\ref{sec:disc_setting}, we 
do the same for the discrete setting. \rev{In particular, we
introduce various (nondimensional) approximation and divergence-conformity factors 
to be used in the analysis. These factors are important to support our claim
that the smallness condition on the mesh size made in the error analysis essentially
behaves (in terms of frequency) as the corresponding condition for the conforming
approximation.} In Section~\ref{sec:dG}, we \rev{introduce the dG approximation
in a rather general setting and show that the setting covers, in particular, the well-known
interior penalty approach.} Moreover, we establish \rev{in Lemma~\ref{lem:weak_cons}}
a key estimate on the weak consistency of the dG approximation. 
In Section~\ref{sec:a_priori}, we deal with the a priori error analysis and
inf-sup stability. 
The main results in this section are Theorem~\ref{th:est_err} and
Theorem~\ref{th:inf_sup}. In Section~\ref{sec:a_post}, we perform the 
a posteriori residual-based error analysis. The main results in this 
section are Theorem~\ref{th:reliability} and Theorem~\ref{th:efficiency}.
Finally, in Section~\ref{sec:bnd_app_fac}, we establish bounds on the 
approximation and divergence conformity factors. These bounds prove that these factors
tend to zero with the mesh size.

\section{Continuous setting}
\label{sec:continuous}

In this section, we briefly recall the functional setting for the time-harmonic
Maxwell's equations and formulate the model problem.

\subsection{Functional spaces}

We use standard notation for Lebesgue and Sobolev spaces. 
To alleviate the notation,
the inner product and associated norm in the spaces $\Ldeux$ and $\Ldeuxd$
are denoted by $(\SCAL,\SCAL)$ and $\|\SCAL\|$, respectively. The 
material properties $\eps$ and $\bnu\eqq \bmu^{-1}$ are measurable functions
that take symmetric positive-definite values in
$\Dom$ with eigenvalues uniformly bounded from above and from below away from zero.
It is convenient to introduce  
the inner product and associated norm weighted by either $\eps$ or $\bnu$,
leading to the notation $(\SCAL,\SCAL)_\eps$, $\|\SCAL\|_\eps$, $(\SCAL,\SCAL)_{\bnu}$
and $\|\SCAL\|_{\bnu}$. Whenever no confusion can arise, we use
the symbol $^\perp$ to denote orthogonality with respect to the inner product
$(\SCAL,\SCAL)_\eps$. Moreover, all the projection operators denoted using the 
symbol $\bPi$ are meant to be
orthogonal with respect to this inner product; we say that the projections are 
$\bL^2_\eps$-orthogonal.

We consider the Hilbert Sobolev spaces
\begin{subequations} \label{eq:Hrot_spaces} \begin{align}
\Hrot & \eqq \{\bv \in \Ldeuxd  \st \ROT\bv\in \Ldeuxd\},\\
\Hrotrotz &\eqq \{\bv \in \Hrot \st \ROT\bv=\bzero\},   \\
\Hrotz &\eqq \{\bv \in \Hrot \st \gamma\upc_\front(\bv)=\bzero\}, \\
\Hrotzrotz &\eqq \{\bv \in \Hrotz \st \ROTZ\bv=\bzero\},          
\end{align} \end{subequations}
where the tangential trace operator
$\gamma\upc_{\front}:\Hrot\rightarrow \bH^{-\frac12}(\front)$ is
the extension by density of the tangent trace operator such
that $\gamma\upc_\front(\bv)=\bv|_{\front}\CROSS \bn_\Dom$ for smooth fields.
The subscript ${}_0$ indicates the curl operator acting on fields
respecting homogeneous Dirichlet conditions. 
Notice that $\ROT$ and $\ROTZ$ are adjoint to each other, i.e.,
$(\ROTZ\bv,\bw)=(\bv,\ROT\bw)$ for all
$(\bv,\bw)\in \Hrotz\times\Hrot$. 
We equip the space $\Hrot$ and its subspaces defined in~\eqref{eq:Hrot_spaces} 
with the following (dimensionally consistent) energy norm:
\begin{equation} \label{eq:energy_norm}
\tnorm{\bv}^2 \eqq \omega^2\|\bv\|_{\eps}^2 + \|\ROT \bv\|_{\bnu}^2.
\end{equation}

We consider the subspace
\begin{equation} 
\bX\upc_0\eqq \Hrotz\cap \Hrotzrotz^\perp, 
\end{equation}
and we introduce the $\bL^2_\eps$-orthogonal projection
\begin{equation}
\bPi\upc_0: \Ldeuxd \to \Hrotzrotz.
\end{equation}
Since $\GRAD\Hunz \subset \Hrotzrotz$, any field $\bxi \in \bX\upc_0$ 
is such that $\DIV(\eps\bxi)=0$ in $\Dom$. Hence,
$\bX\upc_0$ compactly embeds into $\Ldeuxd$ \cite{Weber:80}.

\begin{remark}[Topology of $\Dom$]
We have $\Hrotzrotz^\perp\subset \{\bv\in \Ldeuxd, \; \DIV(\eps\bv)=0\}$ 
with equality if only if $\front$ is connected (see, e.g., \cite{AmBDG:98}).
\end{remark}

\subsection{Model problem}

Given a positive real number $\omega>0$ and a source term 
$\bJ \in (\Hrotz)'$ (the topological dual space of $\Hrotz$),
the model problem amounts to finding $\bE\in\Hrotz$ such that
\begin{equation} \label{eq:weak}
b(\bE,\bw) = \langle\bJ,\bw\rangle \qquad \forall \bw \in \Hrotz,
\end{equation}
with the bilinear form defined on $\Hrotz\times\Hrotz$ such that
\begin{equation}
b(\bv,\bw) \eqq -\omega^2(\bv,\bw)_\eps + (\ROTZ\bv,\ROTZ\bw)_{\bnu},
\end{equation} 
and where the brackets on the right-hand side of~\eqref{eq:weak} denote the duality
product between $(\Hrotz)'$ and $\Hrotz$.
In what follows, we assume that $\omega^2$ is not an eigenvalue of the 
$\eps^{-1}\ROT(\bnu\ROTZ {\cdot})$ operator in $\Dom$. As a result,
the model problem~\eqref{eq:weak} is well-posed.
We observe that the bilinear form $b$ satisfies $|b(\bv,\bw)|\le \tnorm{\bv}\tnorm{\bw}$.
The following inf-sup stability result is established in \cite[Lemma 2]{ThCFE:23}.

\begin{lemma}[Inf-sup stability] \label{lem:infsup}
The following holds:
\begin{equation} \label{eq:infsup_exact}
\frac{1}{1+2\bst} \le \inf_{\substack{\bv \in \Hrotz\\ \tnorm{\bv}=1}} \sup_{\substack{\bw\in \Hrotz \\ \tnorm{\bw}=1}} |b(\bv,\bw)| \le \frac{1}{\bst},
\end{equation}
with the (nondimensional) stability constant
\begin{equation} \label{eq:def_bst}
\bst \eqq
\sup_{\substack{\bg \in \Hrotzrotz^\perp \\ \|\bg\|_\eps = 1}} \omega \tnorm{\bv_{\bg}}.
\end{equation}
Here, for all $\bg \in \Ldeuxd$, $\bv_{\bg}\in \Hrotz$ denotes the unique
solution to~\eqref{eq:weak} with right-hand side $(\bg,\bw)_\eps$,
i.e., $b(\bv_\bg,\bw)=(\bg,\bw)_\eps$ for all $\bw\in \Hrotz$.
\end{lemma}

\section{\rev{Discrete setting}}
\label{sec:disc_setting}

In this section, we introduce the discrete setting and \rev{various
important tools, such as the discrete curl operator, two approximation norms and
a nonconformity measure, and some important
approximation and divergence conformity factors. All these tools are of broader
interest than the dG method presented in the next section; they can indeed be 
applied to analyze other nonconforming approximation methods.}

\subsection{\rev{Mesh and polynomial spaces}}

Let $\calT_h$ be an affine simplicial mesh covering $\Dom$ exactly.
A generic mesh cell is denoted $K$, its diameter $h_K$ and its outward unit normal
$\bn_K$. We define the piecewise constant functions $\th$ and $\tnu$ such that 
\begin{equation} \label{eq:def_th_tnu}
\th|_K\eqq h_K, \qquad \tnu|_K \eqq \min_{\substack{\bu \in \mathbb R^d \\ |\bu| = 1}}
\bnu|_K \bu \cdot \bu,
\qquad \forall K\in \calT_h.  
\end{equation}

We write $\calFh$ for the set of mesh faces, $\calFhi$
for the subset of mesh interfaces (shared by two distinct mesh
cells, $K_l$, $K_r$), and $\calFhb$ for the subset of mesh boundary
faces (shared by one mesh cell, $K_l$, and the boundary, $\front$). 
Every mesh interface $F\in\calFhi$ is oriented
by the unit normal, $\bn_F$, pointing from $K_l$ to $K_r$
(the orientation is arbitrary, but fixed).
Every boundary face $F\in\calFhb$ is oriented by the unit normal $\bn_F\eqq\bn_\Dom|_F$. 
For all $K\in \calT_h$, $\calFK$ is the collection of the mesh faces composing $\partial K$.

Let $k\ge1$ be the polynomial degree. Let $\polP_{k,d}$ be the space
composed of $d$-variate polynomials of total degree at most $k$ and set
$\bpolP_{k,d}\eqq [\polP_{k,d}]^d$. The dG approximation hinges on the
following broken polynomial space:
\begin{equation}
\Pkb \eqq \bset \bv_h\in
\bL^2(\Dom) \st \bv_h|_K\in \bpolP_{k,d},\, \forall K\in\calT_h\eset. \label{eq:def_Vhb}
\end{equation}
Moreover, the error analysis makes use of the following subspaces:
\begin{subequations} \begin{align}
\Pkbrot & \eqq \Pkb \cap \Hrot,\\
\Pkbrotz & \eqq \Pkb \cap \Hrotz,\\
\Pkbrotzrotz & \eqq \Pkb \cap \Hrotzrotz.
\end{align} \end{subequations}%
The superscript ${}\upc$ in the above subspaces is meant to remind us that
all these subspaces are $\Hrot$-conforming. The $\bL^2_{\eps}$-orthogonal projection 
\begin{equation}
\bPi\upc_{h0} : \Ldeuxd \to \Pkbrotzrotz,
\end{equation}
plays a key role in what follows. In particular, we introduce the subspace
\begin{equation} \label{eq:disc_involution}
\bX\upb_{h}\eqq \Pkb \cap \Pkbrotzrotz^\perp,
\end{equation}
which is composed of fields $\bv_h\in \Pkb$ such that $\bPi\upc_{h0}(\bv_h)=\bzero$.

\subsection{Jumps and discrete curl operator}
For all $K\in\calT_h$, all $F\in\calFK$, 
and all $\bv_h\in\Pkb$, we define 
the local trace operators such that 
$\gamma\upg_{K,F}(\bv_h)(\bx) \eqq \bv_h|_K(\bx)$, 
$\gamma\upc_{K,F}(\bv_h)(\bx)\eqq \bv_h|_{K}(\bx)\CROSS\bn_F$, for a.e.~$\bx\in F$.
Then, for all $F\in\calFhi$ and $\mathrm{x}\in\{\mathrm{g},\mathrm{c}\}$,  
we define the jump and average operators such that 
\begin{equation} 
\jump{\bv_h}\upx_F\eqq \gamma\upx_{K_l,F}(\bv_h)-\gamma\upx_{K_r,F}(\bv_h), \quad
\avg{\bv_h}\upx_F\eqq \frac12\big(\gamma\upx_{K_l,F}(\bv_h)+\gamma\upx_{K_r,F}(\bv_h)\big).
\end{equation}
We also set $\jump{\bv_h}\upx_F \eqq \avg{\bv_h}\upx_F \eqq \gamma\upx_{K_l,F}(\bv_h)$ 
for all $F\in\calFhb$.

For every field $\bv_h\in \Pkb$, $\ROTh\bv_h$
denotes the broken curl of $\bv_h$ (evaluated cellwise). Let
$\ell\ge k-1\ge0$. 
We define the discrete curl operator
$\bChz: \bP\upb_k(\calT_h)\to \bP\upb_{\ell}(\calT_h)$ such that, for all $\bv_h\in \Pkb$,
\begin{equation}
\bChz(\bv_h) \eqq \ROTh \bv_h + \bLhz(\bv_h),
\end{equation}
where the jump lifting operator $\bLhz(\bv_h) \in \bP\upb_{\ell}(\calT_h)$
is defined by requiring that
\begin{equation} \label{eq:def_Lh}
(\bLhz(\bv_h),\bphi_h) \eqq \sum_{F\in\calFh} (\jump{\bv_h}\upc_F,\avg{\bphi_h}\upg_F)_{\bL^2(F)}
\end{equation}
for all $\bphi_h\in \bP\upb_{\ell}(\calT_h)$.
Taking the polynomial degree $\ell$ larger than $k-1$ is useful to improve the consistency
property of the discrete curl operator; see Lemma~\ref{lem:weak_cons} below.

It is convenient to introduce the infinite-dimensional space
\begin{equation} \label{eq:def_bVsh}
\bVsh \eqq \Hrotz + \bP\upb_k(\calT_h), 
\end{equation}
where the error $(\bE-\bE_h)$ lives.
Although the sum in~\eqref{eq:def_bVsh} is not direct, any field $\bv_h\in \Hrotz \cap \Pkb$ satisfies $\jump{\bv_h}\upc_F=\bzero$ for all $F\in\calFh$, as well as $\bChz(\bv_h)=\ROTZ \bv_h$. It is therefore legitimate to extend the curl and jump operators to $\bVsh$ by setting, for all $\bv=\tilde \bv + \bv_h\in \bVsh$ with $\tilde\bv \in \Hrotz$ and $\bv_h\in \Pkb$,
\begin{equation} \label{eq:ext_Ch_j}
\bChz(\bv)\eqq\ROTZ\tilde\bv+\ROTh\bv_h+\bLhz(\bv_h), \qquad
\jump{\bv}\upc_F\eqq\jump{\bv_h}\upc_F.
\end{equation}

\subsection{\rev{Approximation norms and nonconformity measure}}

We consider the following two norms:
\begin{subequations} \begin{alignat}{2}
\tnormE{\bv_h}^2 & \eqq \|\tnu^{\frac12} \th^{-1}\bv_h\|^2 + \|\bChz(\bv_h)\|_\bnu^2 &\qquad &\forall \bv_h\in \Pkb, \label{eq:tnormE} \\
\tnormH{\bv}^2 & \eqq \|\bv\|_{\bnu^{-1}}^2 + \|\tnu^{-\frac12} \th \ROT\bv\|^2
&\qquad&\forall \bv \in \Hrot. \label{eq:tnormadj}
\end{alignat} \end{subequations}
(Recall that $\th$ and $\tnu$ are defined in~\eqref{eq:def_th_tnu}.) 
We introduce the following nonconformity measure:
\begin{equation} \label{eq:def_JMP}
\snormnc{\bv} \eqq \min_{\bv_h\upc \in \Pkbrotz} \tnormE{\bv_h-\bv_h\upc} \qquad \forall \bv:=\tilde\bv+\bv_h\in\bVsh. 
\end{equation}
Notice that the definition~\eqref{eq:def_JMP} is independent of the decomposition $\bv:=\tilde\bv+\bv_h$.
\rev{The $\tnormE{\SCAL}$-norm is used to measure the nonconformity in~\eqref{eq:def_JMP},
whereas the $\tnormH{\SCAL}$-norm is used in the next section to estimate the 
approximability properties of some dual solution.}

\subsection{\rev{Approximation and divergence conformity factors}}
\label{sec:def_fac}

Here, we introduce three factors to be used in the error analysis. 
We prove in Section~\ref{sec:bnd_app_fac} that these factors tend to zero
(possibly with a certain rate) as the mesh size tends to zero.

For all $\btheta \in \Hrotzrotz^\perp$, we consider the adjoint problem
consisting of finding $\bzeta_\btheta \in \Hrotz$ such that
\begin{equation} \label{eq:adjoint}
b(\bw,\bzeta_\btheta) = (\bw,\btheta)_\eps \qquad \forall \bw \in \Hrotz.
\end{equation}
Taking any test function $\bw \in \Hrotzrotz\subset \Hrotz$ shows that 
\begin{equation}
\omega^2(\bw,\bzeta_\btheta)_\eps=b(\bw,\bzeta_\btheta)=(\bw,\btheta)_\eps=0,
\end{equation} 
where the first equality follows from $\ROTZ\bw=\bzero$, the second from the
definition of the adjoint solution, and the third from the assumption 
$\btheta \in \Hrotzrotz^\perp$.
Since $\bw$ is arbitrary in $\Hrotzrotz$, this proves that
$\bzeta_\btheta \in \Hrotzrotz^\perp$. Thus, $\bzeta_\btheta \in \bX\upc_0$.
We introduce the (nondimensional) approximation factors%
\begin{subequations} \begin{align}
\gp&\eqq \sup_{\substack{\btheta \in \Hrotzrotz^\perp\\ \|\btheta\|_\eps=1}} \min_{\bv_h\upc\in \Pkbrotz}
\omega \tnorm{\bzeta_\btheta-\bv_h\upc}, \label{eq:def_gp}\\
\gd&\eqq\sup_{\substack{\btheta \in \Hrotzrotz^\perp\\ \|\btheta\|_\eps=1}} \min_{\bPhi\upc_h\in \Plbrot}
\omega \tnormH{\bnu \ROTZ\bzeta_\btheta-\bPhi\upc_h}.\label{eq:def_gd}
\end{align} \end{subequations}
\rev{The approximation factor $\gp$ uses the triple norm $\tnorm{\SCAL}$ defined in~\eqref{eq:energy_norm}, whereas $\gd$ uses the norm $\tnormH{\SCAL}$ defined in~\eqref{eq:tnormadj}}.
Finally, we introduce the (nondimensional) divergence conformity factor
\begin{equation}
\label{eq_gamma_bX_DG}
\gdiv \eqq \sup_{\substack{
\bv_h \in \bX\upb_h
\\
\|\bChz(\bv_h)\|_\bnu^2 + \snormnc{\bv_h}^2} = 1}
\omega \|\bPi\upc_0(\bv_h)\|_\eps.
\end{equation}
Loosely speaking, $\gdiv$ measures how much discretely divergence-free fields
depart from being exactly divergence-free.

\section{\rev{Discontinuous Galerkin approximation}}
\label{sec:dG}

In this section, we formulate the
dG approximation of the model problem~\eqref{eq:weak} \rev{in a rather general
setting and show that the interior penalty dG method fits the proposed framework.
We then examine the Galerkin orthogonality and weak consistency
property of the proposed dG method. We assume from now on that $\bJ\in\Ldeuxd$;
notice though that we do not make any further assumption on $\DIV\bJ$
for the a priori error analysis. Moreover, the sole assumption on the material properties
is uniform boundedness from above and from below away from zero.
The main novel result in this section is Lemma~\ref{lem:weak_cons} which leads to
a weak consistency property of the dG method without requiring any additional regularity
property on the exact solution.}

\subsection{\rev{Stabilization and extended bilinear form}}

We consider a stabilization bilinear form $s_\sharp$ 
defined on $\bVsh\times\bVsh$ for which
we make the following assumptions: 
\begin{subequations} \label{eq:ass_sh} \begin{alignat}{2}
&\textup{(i)}&\;&\text{$s_\sharp$ is symmetric positive semidefinite}, \label{eq:sh_sym} \\
&\textup{(ii)}&\;&s_\sharp(\bv,\SCAL)=s_\sharp(\SCAL,\bv)=0 \quad \forall\bv\in\Hrotz. \label{eq:ker_sh}
\end{alignat} \end{subequations}
\rev{(Notice that the first equality in~\eqref{eq:ker_sh} follows from~\eqref{eq:sh_sym}.)
We also assume that
\begin{equation}
\exists \rho>0 \quad\text{s.t.}\quad \rho \snormnc{\bv_h}
\le
s_\sharp(\bv_h,\bv_h)^{\frac12}
\quad \forall \bv_h \in \bV_h, \label{eq:def_rho}
\end{equation}
where the constant $\rho$ is independent of the mesh size and the frequency.
The value of $\rho$ can depend on the mesh shape-regularity and the polynomial degree.
This assumption is needed only for the a priori error analysis, but not for the 
a posteriori analysis. Furthermore, we notice that although the converse bound
$\tilde \rho \tsc{\bv_h} \le \snormnc{\bv_h}$ for some $\tilde\rho>0$ is not
required anywhere in the analysis, it is reasonably to assume it 
to avoid ill-conditioned linear systems. Finally, we notice that, as
usual in dG methods, the stabilization bilinear form $s_\sharp$ is not bounded 
in the $\Hrot$-norm uniformly with respect to the mesh size. Its role is
to deal with the nonconformity of the discretization when handling the curl operator
by enforcing some penalty on the tangential jumps of discrete fields.}

We define the bilinear forms $\bs$ and $\bss$ on $\bVsh\times\bVsh$ such that 
\begin{subequations} \begin{align} 
\bs(\bv,\bw) &\eqq -\omega^2(\bv,\bw)_\eps + (\bChz(\bv),\bChz(\bw))_\bnu , \label{eq:def_bs}\\
\bss(\bv,\bw) &\eqq \bs(\bv,\bw) + s_\sharp(\bv,\bw).\label{eq:def_bss}
\end{align} \end{subequations}
The bilinear form $\bss$ is used to define the discrete problem and perform the 
a priori error analysis; the bilinear form $\bs$ is useful in the a posteriori
error analysis.
Owing to~\eqref{eq:ext_Ch_j} and~\eqref{eq:ker_sh}, we have the following (minimal)
consistency property:
\begin{equation} \label{eq:bh=b}
\bss(\bv,\bw)=\bs(\bv,\bw)=b(\bv,\bw) \qquad \forall \bv,\bw \in \Hrotz.
\end{equation} 
We extend the $\tnorm{\SCAL}$-norm defined in~\eqref{eq:energy_norm} to
$\bVsh$ by setting, for all $\bv\in\bVsh$,
\begin{subequations} \label{eq:def_tnorm} \begin{align}
\tnorms{\bv}^2 & \eqq \omega^2 \|\bv\|_\eps^2 + \|\bChz(\bv)\|_\bnu^2, \\
\tnormss{\bv}^2 & \eqq \tnorms{\bv}^2 + \tsc{\bv}^2, \qquad \tsc{\bv}^2\eqq s_\sharp(\bv,\bv).
\end{align} \end{subequations}
(The definition of $\tsc{\bv}$ is legitimate owing to~\eqref{eq:sh_sym}.)
This leads to the following boundedness properties on the bilinear form $\bss$:
\begin{subequations} \begin{alignat}{2} 
|\bss(\bv,\bw)| &\le \tnormss{\bv}\, \tnormss{\bw}
&\qquad &\forall (\bv,\bw)\in \bVsh\times\bVsh, \label{eq:bnd_bhs}
\\
|\bss(\bv,\bw)| &\le \tnorms{\bv}\, \tnorm{\bw}
&\qquad &\forall (\bv,\bw)\in \bVsh\times\Hrotz. \label{eq:bnd_bhs_Hrot}
\end{alignat} \end{subequations}%

\subsection{\rev{Discrete problem}}
The discrete problem reads as follows: Find $\bE_h\in \Pkb$ such that
\begin{equation} \label{eq:disc_pb}
\bss(\bE_h,\bw_h)=(\bJ,\bw_h)_{\bL^2(\Dom)} \qquad \forall \bw_h\in \Pkb.
\end{equation}
Notice that we use here the assumption that $\bJ\in\Ldeuxd$.

One simple, yet important, observation is that Galerkin orthogonality holds 
true whenever the discrete test functions are required to be $\Hrotz$-conforming.

\begin{lemma}[Galerkin orthogonality \rev{on conforming test functions}] \label{lem:GO}
If $\bE_h$ solves \eqref{eq:disc_pb}, the following holds true:
\begin{equation} \label{eq:GO}
\bss(\bE-\bE_h,\bv_h\upc)=0 \qquad \forall \bv_h\upc \in \Pkbrotz.
\end{equation}
In particular, we have
\begin{equation} \label{eq:GO_perp}
\bPi\upc_{h0}(\bE-\bE_h)=\bzero.
\end{equation}
\end{lemma}

\begin{proof}
The property~\eqref{eq:GO} follows from the definition of $\bss$ which implies that
$\bss(\bE,\bv_h\upc)=b(\bE,\bv_h\upc)=(\bJ,\bv_h\upc)$ for all $\bv_h\upc \in \Pkbrotz$.
Moreover, since $\Pkbrotzrotz\subset \Pkbrotz$, \eqref{eq:GO} implies
that $(\bE-\bE_h,\bw_h)_\eps=0$ for all $\bw_h \in \Pkbrotzrotz$, which proves~\eqref{eq:GO_perp}. 
\end{proof}

\subsection{Example: interior penalty discontinuous Galerkin method}

The classical interior penalty dG formulation for the model problem~\eqref{eq:weak}
is based upon the following discrete bilinear form \cite{PeScM:02}:
For all $\bv_h,\bw_h\in \Pkb$,
\begin{equation} \label{eq:IPDG} \begin{aligned}
b_h(\bv_h,\bw_h) := {}& -\omega^2(\bv_h,\bw_h)_\eps + (\ROTh \bv_h,\ROTh\bw_h)_\bnu 
+ \eta_* s_h(\bv_h,\bw_h) \\
& + \sum_{F\in\calFh} \big\{ (\avg{\bnu\ROTh\bv_h}\upg_F,\jump{\bw_h}\upc_F)_{\bL^2(F)} 
+ (\jump{\bv_h}\upc_F,\avg{\bnu\ROTh\bw_h}\upg_F)_{\bL^2(F)}\big\}, 
\end{aligned} \end{equation}
with the stabilization bilinear form
\begin{equation} \label{eq:stab_IPDG}
s_h(\bv_h,\bw_h) \eqq \sum_{F\in\calFh} \frac{\tnu_F}{h_F}
(\jump{\bv_h}\upc_F,\jump{\bw_h}\upc_F)_{\bL^2(F)},
\end{equation}
and the user-dependent parameter $\eta_*>0$ is to be taken large enough. 
In~\eqref{eq:stab_IPDG}, $h_F$ denotes the diameter of $F\in\calFh$ and $\tnu_F\eqq \max_{K\in\calT_F} \tnu_K$ with $\calT_F\eqq \{K\in\calT_h\,|\, F\in\calFK\}$. 
\rev{Notice that the extension of $s_h$ to $\bVsh\times\bVsh$ readily
follows from~\eqref{eq:ext_Ch_j}.}

The discrete bilinear form $b_h$ defined in~\eqref{eq:IPDG} can be extended to
$\bVsh\times\bVsh$ using the bilinear form $\bss$ defined in~\eqref{eq:def_bss}
provided the polynomial degree for the jump lifting satisfies $\ell\ge k-1\ge0$
and provided the material property $\bnu$ is piecewise constant on the mesh.
In this situation, $\bss$ is indeed an extension of $b_h$, i.e., 
$\bss|_{\Pkb\times \Pkb}=b_h$, provided the stabilization bilinear form $s_\sharp$ is 
defined as follows:
\begin{equation} \label{eq:def_s_sharp}
s_\sharp(\bv,\bw) \eqq \eta_* s_h(\bv,\bw) 
- (\bLhz(\bv),\bLhz(\bw))_\bnu\quad \forall (\bv,\bw)\in \bVsh\times\bVsh.
\end{equation}
(Notice that $s_\sharp$ is not an extension of $s_h$.) The bilinear form
$s_\sharp$ defined in~\eqref{eq:def_s_sharp} trivially satisfies~\eqref{eq:ker_sh} and is
symmetric. It is positive semidefinite, i.e., \eqref{eq:sh_sym} 
also holds true, if the factor $\eta_*$ is chosen large enough \cite{PerSc:03}. 
The minimal threshold classically depends 
on the mesh shape-regularity and the polynomial degree $\ell$.
\rev{Finally, it is possible to choose the parameter $\eta_*>0$ large enough so that
\eqref{eq:def_rho} holds true. To this purpose, one can, for instance, 
bound $\snormnc{\bv_h}$ by taking $\bv_h\upc\eqq \calI\upcav_{h0}(\bv_h)$ defined using 
the $\Hrotz$-conforming averaging operator analyzed in~\cite{ErnGu:17_quasi}.
We observe in passing that the value of the parameter $\eta_*$ is independent of
the frequency, as it is only related to the nonconformity in the discretization of the 
curl operator.}

\begin{remark}[Weighted averages]
Whenever the jumps of (the eigenvalues of) $\bnu$ are large across the mesh interfaces,
it can be useful to consider weighted $\bnu$-dependent averages to evaluate the 
last two terms on the right-hand side of~\eqref{eq:IPDG}. We refer the reader, e.g., to
\cite{ErStZ:09} for an example in the context of scalar diffusion problems. 
Such weighted averages can be handled in our framework by modifying the
definition of the lifting operator accordingly.
\end{remark}

\subsection{\rev{Weak consistency}}

We now consider the consistency error produced by the discrete curl operator when
tested against general fields. For all $\bPsi \in \Hrotz$ such that $\bnu \ROTZ \bPsi \in \Hrot$
and for all $\bv_h \in \Pkb$, we define the weak consistency error on the discrete curl as
\begin{equation} \label{eq:def_adj_error}
\deltawkc(\bv_h,\bPsi) \eqq (\bv_h,\ROT(\bnu\ROTZ\bPsi)) - (\bChz(\bv_h),\ROTZ\bPsi)_\bnu.
\end{equation}
\rev{The weak consistency error $\deltawkc$ allows us to measure the consistency defect
of the discrete primal problem~\eqref{eq:disc_pb} and the adjoint problem~\eqref{eq:adjoint}.
(Compare with Lemma~\ref{lem:GO} for the consistency
of the discrete primal problem restricted to conforming test functions).

\begin{lemma}[Weak consistency of primal and dual problems]
If $\bE_h$ solves the discrete primal problem~\eqref{eq:disc_pb}, the following holds true:
\begin{subequations} \begin{equation} \label{eq:GOb}
\bss(\bE-\bE_h,\bw_h)=- \deltawkc(\bw_h,\bE) \qquad \forall \bw_h \in \Pkb.
\end{equation}
If $\bzeta_\btheta \in \Hrotz$ solve the adjoint problem~\eqref{eq:adjoint}
with data $\btheta \in \Hrotzrotz^\perp$, the following holds true:
\begin{equation} \label{eq:GOd}
\bss(\bw_h,\bzeta_\btheta)-(\bw_h,\btheta)_\eps= - \deltawkc(\bw_h,\bzeta_\btheta)
\qquad \forall \bw_h \in \Pkb.
\end{equation} \end{subequations}
\end{lemma}

\begin{proof}
Recall the definition~\eqref{eq:def_bss} of $\bss$ and the 
assumption~\eqref{eq:ker_sh} on $s_\sharp$.
To prove~\eqref{eq:GOb}, we observe that
\begin{align*}
\bss(\bE-\bE_h,\bw_h)-\bss(\bE_h,\bw_h) 
&= \bss(\bE,\bw_h)-(\bJ,\bw_h) \\
&= \bs(\bE,\bw_h)-(\bJ,\bw_h) \\
&=(\ROTZ\bE,\bChz(\bw_h))_\bnu - (\ROT(\bnu\ROTZ\bE),\bw_h) \\
&= - \deltawkc(\bw_h,\bE).
\end{align*}
To prove~\eqref{eq:GOd}, we observe that
\begin{align*}
\bss(\bw_h,\bzeta_\btheta) &=\bs(\bw_h,\bzeta_\btheta)=-\omega^2(\bw_h,\bzeta_\btheta)_\eps + (\bChz(\bw_h),\ROTZ\bzeta_\btheta)_\bnu \\
&=(\bw_h,-\omega^2\eps\bzeta_\btheta+\ROT(\bnu\ROTZ\bzeta_\btheta))- \deltawkc(\bw_h,\bzeta_\btheta) \\
&=(\bw_h,\btheta)_\eps - \deltawkc(\bw_h,\bzeta_\btheta).
\end{align*}
This completes the proof.
\end{proof}

The above result motivates the need to bound the weak consistency error on the discrete curl.}%

\begin{lemma}[Weak consistency]
\label{lem:weak_cons}
For all $\bPsi \in \Hrotz$ such that $\bnu \ROTZ \bPsi \in \Hrot$ and for all $\bv_h \in \Pkb$,
the following holds true:
\begin{equation} \label{eq:adj_cons} 
|\deltawkc(\bv_h,\bPsi)| \leq
\snormnc{\bv_h} \min_{\bPhi\upc_h \in \Plbrot}
\tnormH{\bnu \ROTZ \bPsi-\bPhi\upc_h},
\end{equation}
with $\Plbrot \eqq \bP\upb_\ell(\calT_h)\cap \Hrot$.
\end{lemma}

\begin{proof}
We follow the idea of \cite{Chaumont:23_ipdghelmholtz} for the Helmholtz problem.
For any field $\bPhi\upc_h \in \Plbrot$, integration by parts gives
\begin{equation*}
(\bv_h,\ROT\bPhi\upc_h)
=
(\bChz(\bv_h),\bPhi\upc_h).
\end{equation*}
We infer that
\[
\deltawkc(\bv_h,\bPsi) = (\bv_h,\ROT(\bnu\ROTZ\bPsi-\bPhi\upc_h)) - (\bChz(\bv_h),\bnu\ROTZ\bPsi-\bPhi\upc_h).
\]
Let $\bv_h\upc\in \Pkbrotz$ and observe that
$\bChz(\bv_h\upc) = \ROTZ \bv_h\upc$. Integration by parts using 
$\bzeta \eqq \bnu\ROTZ\bPsi-\bPhi\upc_h \in \Hrot$ gives
\[
(\bv_h\upc,\ROT\bzeta) = (\ROTZ\bv_h\upc,\bzeta) = (\bChz(\bv_h\upc),\bzeta).
\]
Putting everything together yields
\begin{align*}
\deltawkc(\bv_h,\bPsi)
&=
(\bv_h-\bv_h\upc,\ROT(\bnu\ROTZ\bPsi-\bPhi\upc_h)) - (\bChz(\bv_h-\bv_h\upc),\bnu\ROTZ\bPsi - \bPhi\upc_h)
\\
&\leq
\tnormE{\bv_h-\bv\upc_h}\tnormH{\bnu\ROTZ\bPsi-\bPhi\upc_h},
\end{align*}
where we used the Cauchy--Schwarz inequality and the 
definitions~\eqref{eq:tnormE} and~\eqref{eq:tnormadj} of the norms $\tnormE{\SCAL}$
and $\tnormH{\SCAL}$, respectively.
Taking the infimum over $\bv_h\upc\in \Pkbrotz$ and over $\bPhi\upc_h \in \Plbrot$
completes the proof.
\end{proof}

\section{A priori error analysis and inf-sup stability}
\label{sec:a_priori}

This section is devoted to the error analysis of the dG approximation.
As usual with Schatz-like arguments, we first establish an error estimate
by assuming that the discrete solution $\bE_h$ exists and then
we prove that the discrete problem~\eqref{eq:disc_pb} is indeed well-posed if $h$ is small enough.

\subsection{Error decomposition and best approximation}
\label{sec:err_dec}

We define the approximation error $\be \eqq  \bE-\bE_h$ and consider the error decomposition 
\begin{equation} \label{eq:def_e_theta}
\be = \btheta_0 + \btheta_\Pi,
\end{equation}
with 
\begin{equation}
\btheta_0\eqq(I-\bPi\upc_0)(\be)\in \Hrotzrotz^\perp, \qquad 
\btheta_\Pi\eqq\bPi\upc_0(\be) \in \Hrotzrotz.
\end{equation}
Let us define the bilinear forms $\bs^+$ and $\bss^+$ on $\bVsh\times \bVsh$ such that
\begin{subequations} \begin{align} 
\bs^+(\bv,\bw) &\eqq \omega^2(\bv,\bw)_\eps + (\bChz(\bv),\bChz(\bw))_\bnu, \label{eq:def_bss_+} \\
\bss^+(\bv,\bw) &\eqq \bs^+(\bv,\bw) + s_\sharp(\bv,\bw). \label{eq:def_bs_+} 
\end{align} \end{subequations}%
The difference with $\bs$ and $\bss$ lies in the sign of the zero-order term.
We define the best-approximation operator
$\bestb: \bVsh \rightarrow \Pkb$ as follows: For
all $\bv \in \bVsh$, $\bestb(\bv) \in \Pkb$ is such that
\begin{equation} \label{eq:def_best_b}
\bss^+(\bv-\bestb(\bv),\bw_h) = 0 \quad \forall \bw_h\in \Pkb.
\end{equation}
The best-approximation error is defined to be
\begin{equation}
\beeta \eqq \bE - \bestb(\bE).
\end{equation}

\begin{lemma}[Properties of $\bestb$] \label{lem:bestb}
The best-approximation operator $\bestb$ defined in~\eqref{eq:def_best_b} enjoys the following two properties:
\begin{subequations} \begin{alignat}{2}
&\tnormss{\bestb(\bv)} \le \tnormss{\bv},&\qquad&\forall \bv\in \bVsh, \label{eq:stab_best_h} \\
&\bestb(\bv) \in \bX\upb_h,&\qquad&\forall \bv \in \Pkbrotzrotz^\perp. \label{eq:useful_pty_b}
\end{alignat} 
In particular, the error $\be=\bE-\bE_h$ satisfies
\begin{equation} \label{eq:bestb_e_0}
\bestb(\be) \in \bX\upb_h.
\end{equation}
\end{subequations}
\end{lemma}

\begin{proof}
\eqref{eq:stab_best_h} follows from the fact that the bilinear form $\bss^+$
is the inner product associated with the $\tnormss{\SCAL}$-norm. To prove~\eqref{eq:useful_pty_b}, 
consider any $\bv \in \Pkbrotzrotz^\perp$. Take any $\bw_h \in \Pkbrotzrotz$
in~\eqref{eq:def_best_b} and observe that $\bChz(\bw_h)=\bzero$ and $s_\sharp(\SCAL,\bw_h)=0$.
Since
\[
\omega^2(\bestb(\bv),\bw_h)_\eps = \omega^2(\bestb(\bv)-\bv,\bw_h)_\eps 
= \bss^+(\bestb(\bv)-\bv,\bw_h) = 0,
\]
we infer that $\bestb(\bv) \in \Pkbrotzrotz^\perp$. Moreover, 
$\bestb(\bv)\in \Pkb$ by construction. This proves~\eqref{eq:useful_pty_b}.
Finally, \eqref{eq:bestb_e_0} follows from~\eqref{eq:GO_perp}
and~\eqref{eq:useful_pty_b}.
\end{proof}

\subsection{Preliminary bounds}

\begin{lemma}[Bound on $\btheta_0$] \label{lem:bnd_thet1}
We have
\begin{equation}
\label{eq:bnd_thet1}
\omega \|\btheta_0\|_\eps \le \gp\, \tnorms{\be}+\gd\,\snormnc{\be},
\end{equation}
with the approximation factors $\gp$ and $\gd$ defined in~\eqref{eq:def_gp}
and~\eqref{eq:def_gd}, respectively.
\end{lemma}

\begin{proof}
Let $\bzeta_\btheta\in \Hrotz$ solve 
the adjoint problem~\eqref{eq:adjoint} with data $\btheta\eqq\btheta_0$.
Since $\bE,\bzeta_\btheta\in \Hrotz$, we infer from \eqref{eq:bh=b} and the definition of the
adjoint solution that
\[
\bss(\bE,\bzeta_\btheta)=b(\bE,\bzeta_\btheta)=(\bE,\btheta_0)_\eps.
\]
Owing to~\eqref{eq:GOd}, we infer that
\begin{align*}
\bss(\bE_h,\bzeta_\btheta) &=(\bE_h,\btheta_0)_\eps - \deltawkc(\bE_h,\bzeta_\btheta).
\end{align*}
Combining the above two identities and using $(\btheta_0,\btheta_\Pi)_\eps=0$ gives
\begin{equation} \label{eq:theta_eps_sq}
\omega \|\btheta_0\|_\eps^2 = \omega (\be,\btheta_0)_\eps =
\omega \bss(\be,\bzeta_\btheta) - \omega\deltawkc(\bE_h,\bzeta_\btheta).
\end{equation}
Owing to Galerkin orthogonality on conforming test functions (see~\eqref{eq:GO}), we infer that, 
for all $\bv_h\upc\in \Pkbrotz$,
\begin{equation} 
\omega \|\btheta_0\|_\eps^2 =
\omega \bss(\be,\bzeta_\btheta-\bv_h\upc) - \omega\deltawkc(\bE_h,\bzeta_\btheta).
\label{eq:id_theta}
\end{equation}
Invoking the boundedness property~\eqref{eq:bnd_bhs_Hrot}
on $\bss$, and using the definition of the approximation factor $\gp$ gives
\[
\omega|\bss(\be,\bzeta_\btheta-\bv_h\upc)|
\le
\tnorms{\be}\, \omega\tnorm{\bzeta_\btheta-\bv_h\upc}
\le \tnorms{\be}\, \gp \|\btheta_0\|_\eps.
\] 
Moreover, invoking Lemma~\ref{lem:weak_cons} to bound the 
weak consistency error and recalling 
the  definition of the approximation factor $\gd$ gives
\[
\omega |\deltawkc(\bE_h,\bzeta_\btheta)|
\le
\snormnc{\bE_h}\, \gd \|\btheta_0\|_\eps
=
\snormnc{\be}\, \gd \|\btheta_0\|_\eps. 
\]
Putting the above two bounds together 
proves \eqref{eq:bnd_thet1}. 
\end{proof}

\begin{lemma}[Bound on $\btheta_\Pi$] \label{lem:bnd_thet2}
We have
\begin{equation} \label{eq:bnd_thet2}
\omega \|\btheta_\Pi\|_\eps \le \omega \|\bPi\upc_0(\beeta)\|_\eps 
+ \gdiv \big\{ \|\bChz(\bestb(\be))\|_\bnu^2 + 
\snormnc{\bestb(\be)}^2 \big\}^{\frac12}.
\end{equation}
\end{lemma}

\begin{proof}
We observe that
\[
\btheta_\Pi + \btheta_0 = \be = \bestb(\be)+(I-\bestb)(\be) = \bestb(\be)+\beeta,
\]
since $(I-\bestb)(\bE_h)=\bzero$. This gives
$\btheta_\Pi =\bestb(\be)+\beeta-\btheta_0.$
Since $(\btheta_\Pi,\btheta_0)_\eps=0$, we infer that
\begin{align*}
\|\btheta_\Pi\|_\eps^2 &= (\btheta_\Pi,\bestb(\be))_\eps + (\btheta_\Pi,\beeta)_\eps \\
&= (\btheta_\Pi,\bPi\upc_0(\bestb(\be)))_\eps + (\btheta_\Pi,\bPi\upc_0(\beeta))_\eps\qqe
\Theta_1 + \Theta_2,
\end{align*}
where we used that $\btheta_\Pi=\bPi\upc_0(\btheta_\Pi)$ and 
that $\bPi\upc_0$ is self-adjoint for the inner product $(\SCAL,\SCAL)_\eps$. 
We bound $\Theta_1$ as follows:
\begin{align*}
|\Theta_1| &\le \|\btheta_\Pi\|_\eps \, \|\bPi\upc_0(\bestb(\be))\|_\eps \\
&\le \|\btheta_\Pi\|_\eps \,\gdiv \omega^{-1} \big\{ \|\bChz(\bestb(\be))\|_\bnu^2 + 
\snormnc{\bestb(\be)}^2 \big\}^{\frac12},
\end{align*}
where we used the divergence conformity factor defined in~\eqref{eq_gamma_bX_DG} 
(this is legitimate since $\bestb(\be) \in \bX\upb_h$ owing to~\eqref{eq:bestb_e_0}). 
Moreover, the Cauchy--Schwarz inequality gives
\[
|\Theta_2| \le \|\btheta_\Pi\|_\eps \, \|\bPi\upc_0(\beeta)\|_\eps.
\]
Putting the above two bounds together proves the assertion.
\end{proof}

\begin{lemma}[Bound on $\tnormss{\btheta_0}$]
\label{lem:bnd_rot_theta}
We have
\begin{align}
\tnormss{\btheta_0}^2 \le {}& \tnormss{(I-\bPi\upc_0)(\beeta)}^2 + 2\omega \|\btheta_\Pi\|_\eps \, \gdiv \big\{ \|\bChz(\bestb(\be))\|_\bnu^2 + 
\snormnc{\bestb(\be)}^2 \big\}^{\frac12} \nonumber \\
&+2\snormnc{\bestb(\be)} \min_{\bPhi\upc_h \in \Plbrot}
\tnormH{\bnu \ROTZ \bE-\bPhi\upc_h} + 4 \omega^2 \|\btheta_0\|_\eps^2. \label{eq:bnd_rot_theta}
\end{align}
\end{lemma}

\begin{proof}
Since $\be=\beeta+\bestb(\be)$ as shown in the above proof, we have
$\btheta_0=(I-\bPi\upc_0)(\be) = (I-\bPi\upc_0)(\beeta) + (I-\bPi\upc_0)(\bestb(\be))$.
This gives
\begin{align*}
\bss(\btheta_0,\btheta_0) &= \bss(\btheta_0,(I-\bPi\upc_0)(\beeta)) +
\bss(\btheta_0,(I-\bPi\upc_0)(\bestb(\be))) \\
&= \bss(\btheta_0,(I-\bPi\upc_0)(\beeta)) +
\bss(\be,(I-\bPi\upc_0)(\bestb(\be))), 
\end{align*}
where the second equality follows from $\bss(\btheta_\Pi,(I-\bPi\upc_0)(\SCAL))=0$.
The first term on the right-hand side is bounded by invoking the continuity property~\eqref{eq:bnd_bhs}, giving
\[
|\bss(\btheta_0,(I-\bPi\upc_0)(\beeta))| \le \tnormss{\btheta_0} \, \tnormss{(I-\bPi\upc_0)(\beeta)}.
\]
The second term is decomposed as
$\bss(\be,(I-\bPi\upc_0)(\bestb(\be)))=\beta_1+\beta_2$ with
\[
\beta_1 := \bss(\be,\bestb(\be)), \qquad
\beta_2 := -\bss(\be,\bPi\upc_0(\bestb(\be))) = \omega^2(\btheta_\Pi,\bPi\upc_0(\bestb(\be)))_\eps.
\]
Recalling~\eqref{eq:GOb}, i.e., the weak consistency of the discrete primal problem for all 
test functions in $\Pkb$, we have $\beta_1 = - \deltawkc(\bestb(\be),\bE)$. Hence, invoking
Lemma~\ref{lem:weak_cons} gives
\[
|\beta_1| \leq
\snormnc{\bestb(\be)} \min_{\bPhi\upc_h \in \Plbrot}
\tnormH{\bnu \ROTZ \bE-\bPhi\upc_h}.
\] 
Moreover, using the Cauchy--Schwarz inequality and since $\bestb(\be) \in \bX\upb_h$
(see~\eqref{eq:bestb_e_0}), we have 
\[
|\beta_2| \le \omega \|\btheta_\Pi\|_\eps \, \gdiv \big\{ \|\bChz(\bestb(\be))\|_\bnu^2 + 
\snormnc{\bestb(\be)}^2 \big\}^{\frac12}.
\]
Altogether, this gives
\begin{align*}
\bss(\btheta_0,\btheta_0) \le {}& \tnormss{\btheta_0} \, \tnormss{(I-\bPi\upc_0)(\beeta)} 
+ \omega \|\btheta_\Pi\|_\eps \, \gdiv \big\{ \|\bChz(\bestb(\be))\|_\bnu^2 + 
\snormnc{\bestb(\be)}^2 \big\}^{\frac12} \\
&+\snormnc{\bestb(\be)} \min_{\bPhi\upc_h \in \Plbrot}
\tnormH{\bnu \ROTZ \bE-\bPhi\upc_h}.
\end{align*}
Since $\tnormss{\btheta_0}^2 = \bss(\btheta_0,\btheta_0) + 2 \omega^2 \|\btheta_0\|_\eps^2$, we infer that
\begin{align*}
\tnormss{\btheta_0}^2 \le {}& \tnormss{\btheta_0} \, \tnormss{(I-\bPi\upc_0)(\beeta)} + \omega \|\btheta_\Pi\|_\eps \, \gdiv \big\{ \|\bChz(\bestb(\be))\|_\bnu^2 + 
\snormnc{\bestb(\be)}^2 \big\}^{\frac12} \\
&+\snormnc{\bestb(\be)} \min_{\bPhi\upc_h \in \Plbrot}
\tnormH{\bnu \ROTZ \bE-\bPhi\upc_h} + 2 \omega^2 \|\btheta_0\|_\eps^2.
\end{align*}
Dealing with the first term on the right-hand side by Young's inequality gives~\eqref{eq:bnd_rot_theta}.
\end{proof}

\subsection{A priori error estimate}

\rev{We are now ready to establish our main result on the a priori error analysis
which establishes asymptotic optimality. Importantly, the (frequency dependent)
constants controlling the
smallness of the mesh size are essentially those appearing in a conforming approximation.}

\begin{theorem}[\rev{Asymptotically optimal error estimate} and discrete well-posedness] \label{th:est_err}
Assume~\eqref{eq:def_rho}. The following holds:
\begin{equation}
(1-c_\gamma)\tnormss{\be}^2 \le (1+4\gdiv)\tnormss{\beeta}^2
+2\rho^{-1}\tnormss{\be} \min_{\bPhi\upc_h \in \Plbrot}
\tnormH{\bnu \ROTZ \bE-\bPhi\upc_h}.
\label{eq:err_est_simple}
\end{equation}
with $c_\gamma := 8\max(\gp^2,\rho^{-2}\gd^2) + \max(1,\rho^{-2})(\gdiv+3\gdiv^2)$.
Consequently, if the mesh size is small enough so that $c_\gamma<1$, 
the discrete problem~\eqref{eq:disc_pb} is well-posed.
\end{theorem}

\begin{proof}
We use~\eqref{eq:bnd_thet2} in~\eqref{eq:bnd_rot_theta} to infer that
\begin{align*}
\tnormss{\btheta_0}^2 \le {}& \tnormss{(I-\bPi\upc_0)(\beeta)}^2 + 2\omega \|\bPi\upc_0(\beeta)\|_\eps \, \gdiv \big\{ \|\bChz(\bestb(\be))\|_\bnu^2 + 
\snormnc{\bestb(\be)}^2 \big\}^{\frac12} \\
&+2 \gdiv^2 \big\{ \|\bChz(\bestb(\be))\|_\bnu^2  + 
\snormnc{\bestb(\be)}^2 \big\} \\
&+2\snormnc{\bestb(\be)} \min_{\bPhi\upc_h \in \Plbrot}
\tnormH{\bnu \ROTZ \bE-\bPhi\upc_h} + 4 \omega^2 \|\btheta_0\|_\eps^2.
\end{align*}
We now square~\eqref{eq:bnd_thet2} and add the result to the above estimate.
Since
\begin{equation*}
\tnormss{\be}^2 = \tnormss{\btheta_0}^2 + \omega^2 \|\btheta_\Pi\|_\eps^2,
\qquad
\tnormss{\beeta}^2 = \tnormss{(I-\bPi\upc_0)(\beeta)}^2 +
\omega^2 \|\bPi\upc_0(\beeta)\|_\eps^2,
\end{equation*}
we obtain
\begin{align*}
\tnormss{\be}^2 \le {}& \tnormss{\beeta}^2 + 4\omega \|\bPi\upc_0(\beeta)\|_\eps \, \gdiv \big\{ \|\bChz(\bestb(\be))\|_\bnu^2 + 
\snormnc{\bestb(\be)}^2 \big\}^{\frac12}  \\
&+3\gdiv^2\big\{ \|\bChz(\bestb(\be))\|_\bnu^2 + 
\snormnc{\bestb(\be)}^2 \big\} \\
&+2\snormnc{\bestb(\be)} \min_{\bPhi\upc_h \in \Plbrot}
\tnormH{\bnu \ROTZ \bE-\bPhi\upc_h} + 4 \omega^2 \|\btheta_0\|_\eps^2.
\end{align*}
We deal with the second term on the right-hand side by Young's inequality. 
Since $\omega\|\bPi\upc_0(\beeta)\|_{\eps} \leq \omega\|\beeta\|_{\eps} \leq \tnormss{\beeta}$, this gives
\begin{align*}
\tnormss{\be}^2 \le {}& (1+4\gdiv)\tnormss{\beeta}^2 
+(\gdiv+3\gdiv^2)\big\{ \|\bChz(\bestb(\be))\|_\bnu^2 + 
\snormnc{\bestb(\be)}^2 \big\} \\
&+2\snormnc{\bestb(\be)} \min_{\bPhi\upc_h \in \Plbrot}
\tnormH{\bnu \ROTZ \bE-\bPhi\upc_h} + 4 \omega^2 \|\btheta_0\|_\eps^2.
\end{align*}
We invoke~\eqref{eq:bnd_thet1} to bound the last term on the right-hand side. This yields
\begin{align*}
\tnormss{\be}^2 \le {}& (1+4\gdiv)\tnormss{\beeta}^2 
+(\gdiv+3\gdiv^2)\big\{ \|\bChz(\bestb(\be))\|_\bnu^2 + 
\snormnc{\bestb(\be)}^2 \big\} \\
&+2\snormnc{\bestb(\be)} \min_{\bPhi\upc_h \in \Plbrot}
\tnormH{\bnu \ROTZ \bE-\bPhi\upc_h} + 8 \big(\gp^2 \tnorms{\be}^2 +
\gd^2 \snormnc{\be}^2\big).
\end{align*}
We observe that
\begin{align*}
\|\bChz(\bestb(\be))\|_\bnu^2 + 
\snormnc{\bestb(\be)}^2 &\le \max(1,\rho^{-2}) \big\{ \|\bChz(\bestb(\be))\|_\bnu^2 + 
\tsc{\bestb(\be)}^2 \big\} \\
&\le \max(1,\rho^{-2}) \tnormss{\bestb(\be)}^2 
\le \max(1,\rho^{-2}) \tnormss{\be}^2,
\end{align*}
where the last bound follows from~\eqref{eq:stab_best_h}. 
Moreover, we have 
\[
\gp^2 \tnorms{\be}^2 +
\gd^2 \snormnc{\be}^2\le \max(\gp^2,\rho^{-2}\gd^2)(\tnorms{\be}^2+\tsc{\be}^2)
= \max(\gp^2,\rho^{-2}\gd^2)\tnormss{\be}^2.
\] 
Combining the above bounds shows that
\[
(1-c_\gamma)\tnormss{\be}^2 \le (1+4\gdiv)\tnormss{\beeta}^2
+2\snormnc{\bestb(\be)} \min_{\bPhi\upc_h \in \Plbrot}
\tnormH{\bnu \ROTZ \bE-\bPhi\upc_h}.
\]
Since $\rho \snormnc{\bestb(\be)} \leq \tsc{\bestb(\be)}
\leq \tnormss{\bestb(\be)} \leq \tnormss{\be}$, this readily gives~\eqref{eq:err_est_simple}.
\end{proof}

\begin{remark}[Error estimate~\eqref{eq:err_est_simple}] \label{rem:err_est}
The last term on the right-hand side of~\eqref{eq:err_est_simple} stems from the 
weak consistency of the discrete formulation and somewhat pollutes the asymptotic
optimality of the a priori error estimate. We notice that this term can be made 
superconvergent already with the choice $\ell=k$ (provided $\bnu\ROTZ\bE$ is smooth
enough). The (slight) price to pay is to choose the stabilization factor $\eta_*$
large enough so that $s_\sharp$ is indeed positive semidefinite for $\ell=k$ 
(see~\eqref{eq:def_s_sharp}).
\end{remark}

\subsection{Inf-sup stability}
\label{sec:infsup}

Here, we establish the discrete inf-sup stability of 
the bilinear form $\bss$ on $\Pkb \times \Pkb$.
\rev{As for the error estimate from Theorem~\ref{th:est_err}, the main insight
is that the (frequency dependent) constants controlling the
smallness of the mesh size are essentially those appearing in a conforming approximation.
The inf-sup stability constant of the discrete problem also depends on the frequency
through the stability constant $\bst$ of the exact problem; again, this is the
same situation as for a conforming approximation.}

\begin{theorem}[Inf-sup stability]
\label{th:inf_sup}
Under assumption~\eqref{eq:def_rho}, we have
\begin{equation} \label{eq:inf_sup}
\min_{\substack{\bv_h \in \Pkb \\ \tnormss{\bv_h} = 1}}
\max_{\substack{\bw_h \in \Pkb \\ \tnormss{\bw_h} = 1}}
|\bss(\bv_h,\bw_h)|
\geq \frac{1-c'_\gamma}{1 + 2\bst},
\end{equation}
with $c'_\gamma \eqq 2\big(\gp + \tfrac12 \rho^{-1} \gd+\max(1,\rho^{-2}) \gdiv^2\big)$.
\end{theorem}

\begin{proof}
Let $\bv_h \in \Pkb$. We build a suitable $\bw_h \in \Pkb$ so that
$\tnormss{\bw_h}\le (1 + 2\bst)\tnormss{\bv_h}$ and
$\bss(\bv_h,\bw_h) \ge (1-c'_\gamma)\tnormss{\bv_h}^2$.

(1) Set $\bv_{h0}\eqq (I-\bPi\upc_{h0})(\bv_h) \in \bX\upb_{h}$
and $\bv_{h\Pi}\eqq \bPi\upc_{h0}(\bv_h)\in \Pkbrotzrotz$, so that 
$\bv_h=\bv_{h0}+\bv_{h\Pi}$. We further decompose $\bv_{h0}$ as
$\bv_{h0} = \bphi_0+\bphi_\Pi$ with $\bphi_0\eqq (I-\bPi\upc_0)(\bv_{h0})$ and 
$\bphi_\Pi\eqq \bPi\upc_0(\bv_{h0})$. Let $\bxi_0\in \Hrotz$ be the unique adjoint 
solution such that $b(\bw,\bxi_0)=\omega^2(\bw,\bphi_0)_\eps$ for all $\bw\in \Hrotz$. 
We set $\bxi_{h0}\eqq \bestc(\bxi_0)$, where $\bestc:\Hrotz \to \Pkbrotz$ is uniquely defined
by requiring that $\bs^+(\bv-\bestc(\bv),\bw_h) = 0$, for all $\bv \in \Hrotz$ and all
$\bw_h\in \Pkbrotz$. Finally, we set 
\[
\bw_h \eqq \bv_{h0} + 2 \bxi_{h0} - \bv_{h\Pi} \in \Pkb.
\]

(2) Upper bound on $\tnormss{\bw_h}$. The same argument as in the proof of Lemma~\ref{lem:bestb} shows that $\tnorm{\bxi_{h0}} \le \tnorm{\bxi_0}$. Moreover, since $\bphi_0\in\Hrotzrotz^\perp$, we have
\[
\bst^{-1} \tnorm{\bxi_0} \le \omega\|\bphi_0\|_\eps \le \omega\|\bv_{h0}\|_\eps \le \omega\|\bv_{h}\|_\eps \le \tnorms{\bv_h} \le \tnormss{\bv_h}.
\]
This gives
\begin{align*}
\tnormss{\bw_h}^2 &= \tnormss{\bv_{h0}+2\bxi_{h0}}^2 + \omega^2\|\bv_{h\Pi}\|_\eps^2 
\le \big( \tnormss{\bv_{h0}} + 2 \tnorm{\bxi_{h0}} \big)^2 + \omega^2\|\bv_{h\Pi}\|_\eps^2\\
&\le (1+2\bst)^2 \tnormss{\bv_{h0}}^2 + \omega^2\|\bv_{h\Pi}\|_\eps^2 
\le (1+2\bst)^2 \tnormss{\bv_{h}}^2.
\end{align*}

(3) Lower bound on $\bss(\bv_h,\bw_h)$. We first observe that
\[
\bss(\bv_h,\bxi_{h0}) = \bss(\bv_{h0},\bxi_{h0})+\bss(\bv_{h\Pi},\bxi_{h0}) = \bss(\bv_{h0},\bxi_{h0}) \\
=\bss(\bv_{h0},\bxi_{0})+\bss(\bv_{h0},\bxi_{h0}-\bxi_0).
\] 
Owing to~\eqref{eq:GOd}, we have
\[
\bss(\bv_{h0},\bxi_{0}) = \omega^2 (\bv_{h0},\bphi_0)_\eps - \deltawkc(\bv_{h0},\bxi_0)
= \omega^2 \|\bphi_0\|_\eps^2 - \deltawkc(\bv_{h0},\bxi_0).
\]
Invoking Lemma~\ref{lem:weak_cons} and the definition~\eqref{eq:def_gd} 
of the approximation factor $\gd$ gives
\[
\bss(\bv_{h0},\bxi_{0}) \ge \omega^2 \|\bphi_0\|_\eps^2 - \snormnc{\bv_{h0}} \gd \omega
\|\bphi_0\|_\eps.
\]
Using the above bound on $\|\bphi_0\|_\eps$ together with $\snormnc{\bv_{h0}} = \snormnc{\bv_{h}}$ and assumption~\eqref{eq:def_rho}, we infer that
\begin{equation} \label{eq:lb_bsharp_1}
\bss(\bv_{h0},\bxi_{0}) \ge \omega^2 \|\bphi_0\|_\eps^2 - \rho^{-1} \gd \tsc{\bv_{h}}
\tnorms{\bv_h} \ge \omega^2 \|\bphi_0\|_\eps^2 - \tfrac12 \rho^{-1} \gd \tnormss{\bv_h}^2,
\end{equation}
where the last bound follows from Young's inequality.
Furthermore, using the definition~\eqref{eq:def_gp} 
of the approximation factor $\gp$ together with the boundedness property~\eqref{eq:bnd_bhs_Hrot} gives
\[
\bss(\bv_{h0},\bxi_{h0}-\bxi_0) \ge - \tnormss{\bv_h} \gp \omega \|\bphi_0\|_\eps
\ge - \gp \tnormss{\bv_h}^2.
\]
Combining this lower bound with~\eqref{eq:lb_bsharp_1}, we infer that
\begin{equation} \label{eq:lb_bsharp_2}
\bss(\bv_{h0},\bxi_{h0}) \ge \omega^2 \|\bphi_0\|_\eps^2
- \big( \gp + \tfrac12 \rho^{-1} \gd \big) \tnormss{\bv_h}^2.
\end{equation}
Furthermore, using the divergence conformity factor $\gdiv$ yields
\begin{align*}
\omega^2\|\bphi_\Pi\|_\eps^2 = \omega^2 \|\bPi\upc_0(\bv_{h0})\|_\eps^2 &\le \gdiv^2
\big\{ \|\bChz(\bv_{h0})\|_\bnu^2 + \snormnc{\bv_{h0}}^2\big\} \\
&\le \max(1,\rho^{-2}) \gdiv^2 \tnormss{\bv_{h0}}^2 
\le \max(1,\rho^{-2}) \gdiv^2 \tnormss{\bv_h}^2.
\end{align*}
Since $\|\bv_{h0}\|_\eps^2=\|\bphi_0\|_\eps^2+\|\bphi_\Pi\|_\eps^2$, combining this bound with~\eqref{eq:lb_bsharp_2} gives
\[
\bss(\bv_{h0},\bxi_{h0}) \ge \omega^2 \|\bv_{h0}\|_\eps^2 - \tfrac12 c'_\gamma \tnormss{\bv_h}^2.
\]
Finally, since $\bss(\bv_h,\bv_{h0}-\bv_{h\Pi})=\tnormss{\bv_h}^2 - 2\omega^2\|\bv_{h0}\|_\eps^2$,
we infer that
\[
\bss(\bv_h,\bw_h) = \bss(\bv_h,\bv_{h0} + 2 \bxi_{h0} - \bv_{h\Pi})
\ge (1-c'_\gamma)\tnormss{\bv_h}^2.
\]
This completes the proof.
\end{proof}

\begin{remark}[Discrete inf-sup constant] 
The discrete inf-sup constant appearing on the left-hand side of~\eqref{eq:inf_sup}
tends to $(1+2\bst)^{-1}$ as the mesh is refined, thus approaching, 
by up to a factor of two at most, the inf-sup constant from the continuous setting.
\end{remark}

\section{A posteriori residual-based error analysis}
\label{sec:a_post}

In this section, we estimate the error $\be:=\bE-\bE_h$ by means of local
residual-based quantities called error indicators. We derive both a global upper error
bound (reliability) and local lower error bounds (local efficiency). 
The only property required for the discrete object $\bE_h$
in the a posteriori error analysis
is to satisfy the Galerkin orthogonality \eqref{eq:GO} on conforming test functions,
i.e., $\bss(\bE-\bE_h,\bv_h\upc) =\bs(\bE-\bE_h,\bv_h\upc) = 0$ for all $\bv_h\upc\in \Pkbrotz$.
Lemma~\ref{lem:GO} shows that the dG solution solving~\eqref{eq:disc_pb} satisfies
this property. For simplicity, we keep the notation $\bE_h$ in this section.
\rev{For the a posteriori error analysis, we assume that $\DIV\bJ\in L^2(D)$ and 
that the material properties are piecewise constant on the mesh.}

\subsection{Notation \rev{and interpolation operators}}

For all $K \in \calT_h$, the element patch $\tK$ (resp., $\chK$, $\Kupf$) denotes the domain
covered by all the cells $K' \in \calT_h$ sharing at least one vertex (resp., edge, face) with
$K$. Similarly, the extended patch $\ttK$ (resp., $\tttK$) is the domain covered by all the
cells $K'' \in \calT_h$ sharing at least one vertex with a cell $K'\subset \tK$
(resp., $K'\subset \ttK$). For a face $F \in \calF_h$, $\tF$ is the domain covered by the one or
two cells sharing $F$. Whenever no confusion can arise, 
we also employ the symbols $\tK$, $\chK$, $\ttK$, $\tttK$, $\tF$ 
for the set of cells covering the domains.
We employ the symbol $\kappa_{\calT_h}$ for the shape-regularity parameter
of the mesh $\calT_h$, and $C(\kappa_{\calT_h})$ denotes any generic constant
solely depending on $\kappa_{\calT_h}$ and whose value can change at each occurrence.
For any subset $\calT \subset \calT_h$, we introduce the notation
\begin{equation}
\begin{aligned}
\varepsilon_{\max,\calT} &:= \max_{K\in\calT}
\max_{\bx \in K} \max_{\substack{\bu \in \mathbb R^d \\ |\bu| = 1}}
\max_{\substack{\bv \in \mathbb R^d \\ |\bv| = 1}}
\eps(\bx) \bu \cdot \bv,
\\
\varepsilon_{\min,\calT} &:= \min_{T\in\calT}
\min_{\bx \in K} \min_{\substack{\bu \in \mathbb R^d \\ |\bu| = 1}}
\eps(\bx) \bu \cdot \bu,
\end{aligned}
\end{equation}
and define $\nu_{\max,\calT}$ and $\nu_{\min,\calT}$ similarly. Then,
$\vel_\calT := (\nu_{\min,\calT}/\varepsilon_{\max,\calT})^{\frac12}$
stands for the minimum velocity in the subdomain covered by the cells
in $\calT$.  We write $\|v\|_{\calT}^2 := \sum_{K\in\calT} \|v\|_{L^2(K)}^2$
and employ a similar notation if $v$ is vector-valued. We also write
$\|v\|_{\calF}^2 := \sum_{F\in\calF} \|v\|_{L^2(F)}^2$
for every subset $\calF \subset \calF_h$. For simplicity, we assume that
$\ell\in\{k-1,k\}$ in the discrete curl operator and do not track the dependency
on $\ell$ of the constants.

We employ the quasi-interpolation operators from \cite{KarMe:15} (see also \cite{Melenk:05}
and see \cite[Corollary~2.5]{DongErn:24} for using the seminorm in the extended patch $\ttK$).
Specifically, there exists an operator
$\Ig: H^1_0(\Dom) \to \calP_{k+1}\upb(\calT_h) \cap H^1_0(\Dom)$
such that, for all $q \in H^1_0(\Dom)$ and all $K \in \calT_h$,
\begin{equation} \label{eq:inter_Ig}
\frac{k^2}{h_K^2}\|q-\Ig (q)\|_K^2
+
\frac{k}{h_K}\|q-\Ig (q)\|_{\partial K}^2
\leq
C(\kappa_{\calT_h}) \|\nabla q\|_{\ttK}^2.
\end{equation}
Similarly, there exists an operator $\bIc: \bH^1_0(\Dom) \to \Pkbrotz$
such that, for all $\bw \in \bH^1_0(\Dom)$ and all $K \in \calT_h$,
\begin{equation}
\label{eq:inter_Ic}
\frac{k^2}{h_K^2}\|\bw-\bIc (\bw)\|_K^2
+
\frac{k}{h_K}\|(\bw-\bIc (\bw)) {\times} \bn_K\|_{\partial K}^2
\leq
C(\kappa_{\calT_h}) \|\nabla \bw\|_{\ttK}^2.
\end{equation}

We will also need the quasi-interpolation averaging operator
$\calI_{h0}\upcav:\Pkb\to \Pkbrotz$ from~\cite{ErnGu:17_quasi} which is such that
there is $\cav \geq 1$ so that, for all $\bv_h\in \Pkb$ and all $K\in\calT_h$, 
\begin{multline}
\label{eq:approx_av}
\frac{k}{h_K} \|\bv_h-\calI_{h0}\upcav(\bv_h)\|_K + \|\ROT(\bv_h-\calI_{h0}\upcav(\bv_h))\|_K 
\le
\\
C(\kappa_{\calT_h}) \cav \bigg( \frac{k^2}{h_K}\bigg)^{\frac12} \bigg\{ \sum_{K'\in \chK} \|\jump{\bv_h}_{\partial K'}\upc\|_{\partial K'}^2\bigg \}^{\frac12}.
\end{multline}
The dependency of $\cav$ on $k$ has been explored in some specific cases for $d=2$.
\cite{BurEr:07,HoScW:07}. We keep this factor here as the analysis of the behavior
of $\cav$ in $k$ goes beyond the present scope. Invoking a discrete trace inequality
(see, e.g., \cite[Lem.~12.10]{EG_volI}) yields
\begin{equation}
\label{eq:bound_lifting}
\|\ROT \bv_h-\bChz (\bv_h)\|_K
=
\|\bLhz(\bv_h)\|_K
\leq
C(\kappa_{\calT_h}) \bigg( \frac{k^2}{h_K}\bigg)^{\frac12} \|\jump{\bv_h}\upc_{\partial K}\|_{\partial K}.
\end{equation}
Combining \eqref{eq:approx_av} and \eqref{eq:bound_lifting} gives
\begin{multline}
\label{eq:approx_av_gcurl}
\frac{k}{h_K} \|\bv_h-\calI_{h0}\upcav(\bv_h)\|_K + \|\bChz(\bv_h-\calI_{h0}\upcav(\bv_h))\|_K 
\le
\\
C(\kappa_{\calT_h}) \cav \bigg( \frac{k^2}{h_K}\bigg)^{\frac12}
\bigg\{ \sum_{K'\in \chK} \|\jump{\bv_h}_{\partial K'}\upc\|_{\partial K'}^2\bigg \}^{\frac12},
\end{multline}
and
\begin{multline}
\label{eq:approx_av_sharp}
\tnorm{\bv_h-\calI_{h0}\upcav(\bv_h)}_{\sharp}
\le
\\
C(\kappa_{\calT_h}) \cav  \left (
1 + \max_{K \in \calT_h} \frac{\omega h_K}{k\vel_{\min,\chK}}
\right ) \bigg\{ \sum_{K \in \calT_h} \nu_{\max,\chK}
\frac{k^2}{h_K} \|\jump{\bv_h}_{\partial K}\upc\|_{\partial K}^2 \bigg\}^{\frac12}.
\end{multline}

\begin{remark}[Broken curl]
In view of \eqref{eq:bound_lifting}, we can freely replace the discrete curl
$\bChz$ by the broken curl in the definition of the estimator $\eta$
and the error measure $\tnorm{{\cdot}}_{\dagger}$.
\end{remark}

\subsection{Estimator and error measure}

The a posteriori error estimator is written as the sum over the mesh cells
of local error indicators $\eta_K$ for all $K\in\calT_h$. The local error indicator 
consists of three pieces. The first two respectively measure the residuals of the
divergence constraint and of Maxwell's equations:
\begin{subequations}
\label{eq_def_eta}
\begin{equation}
\etadK^2
:=
\varepsilon_{\min,\tttK}^{-1}
\Big \{
\frac{h_K^2}{\omega^2 k^2}\|\DIV (\bJ+\omega^2\eps\bE_h)\|_K^2
+
\frac{\omega^2 h_K}{k}\|\jump{\eps\bE_h}_{\partial K}\upd\|_{\partial K \setminus \partial \Omega}^2
\Big \}, \\
\end{equation}
and
\begin{multline}
\etacK^2
:=
\nu_{\min,\tttK}^{-1} \Big \{
\frac{h_K^2}{k^2}\|\bJ+\omega^2 \eps \bE_h-\ROT (\bnu \bChz(\bE_h))\|_K^2
\\
+
\frac{h_K}{k}\|\jump{\bnu \bChz(\bE_h)}_{\partial K}\upc\|_{\partial K \setminus \partial \Omega}^2
\Big \},
\end{multline}
where $\jump{\eps\bE_h}_{\partial K}\upd|_F := \jump{\eps\bE_h}_F\upg\SCAL \bn_F$ and 
$\jump{\bnu \bChkz(\bE_h)}_{\partial K}\upc|_F := \jump{\bnu \bChkz(\bE_h)}_F\upg{\times}\bn_F$
for all $F\in\calF_K$.
The last part of the estimator controls the nonconformity of the discrete field $\bE_h$ as follows:
\begin{equation}
\label{eq_def_etajK}
\etajK^2
:= \rev{\cav}
\frac{\rev{\nu_{\max,\chK}k^2}}{h_K}\|\jump{\bE_h}_{\partial K}\upc\|_{\partial K}^2,
\end{equation}
\end{subequations}
where $\jump{\bE_h}_{\partial K}\upc|_F := \jump{\bE_h}_F\upg{\times}\bn_F$ for all $F\in\calF_K$.
For shortness, we also introduce the following notation:
\begin{equation}
\etaK^2 := \etadK^2+\etacK^2+\etajK^2,
\qquad
\eta_{\bullet}^2 := \sum_{K \in \calT_h} \eta_{\bullet,K}^2,
\qquad
\eta^2 := \sum_{K \in \calT_h} \etaK^2,
\end{equation}
with $\bullet\in\{\operatorname{div},\operatorname{curl},\operatorname{nc}\}$.

For all $\calT \subset \calT_h$, we define the error measure
\begin{align}
\tnorm{\be}_{\dagger,\calT}^2
:= {}&
\sum_{K \in \calT}
\left \{
\omega^2 \|\be\|_{\eps,K}^2
+
\|\bChz(\be)\|_{\bnu,K}^2 + \rev{\cav}
\frac{\rev{\nu_{\max,\chK}k^2}}{h_K}\|\jump{\be}_{\partial K}\upc\|_{\partial K}^2
\right \},
\end{align}
and we omit the subscript $\calT$ whenever $\calT = \calT_h$.
A crucial observation is that the last term in the norm measuring
the nonconformity can be chosen independently of the stabilization
in the dG scheme. In particular, it does not have to be large enough.

\subsection{Error upper bound (reliability)}

We start by controlling the PDE residual in Lemma \ref{lemma_pde_residual}.
Lemma \ref{lemma_pde_residual} is similar to \cite[Lemma 3.2]{chaumontfrelet_vega_2022a},
but the result proposed here is sharper. In particular, the constant only depends on the
shape-regularity of the mesh. Notice also that we consider here only conforming 
test functions so that we can work with the bilinear form $\bs$ rather than $\bss$.

\begin{lemma}[Residual]
\label{lemma_pde_residual}
For all $\bv \in \Hrotz$, we have
\begin{equation}
\label{eq_residual_Hcurl}
|\bs(\be,\bv)|
\leq
C(\kappa_{\calT_h}) \etadc \tnorm{\bv},
\end{equation}
with $\etadc^2 := \etad^2+\etac^2$.
\end{lemma}

\begin{proof}
Here, we invoke \cite[Theorem 1]{Schoberl:07}, which states that, given any
$\bw \in \Hrotz$, there exists $\bS_{h0}\upc(\bw) \in \Pkbrotz$, such that
\begin{equation}
\label{tmp_decomposition_schoberl}
\bw-\bS_{h0}\upc(\bw) = \GRAD q + \bphi,
\end{equation}
with $q \in H^1_0(\Dom)$, $\bphi \in \bH^1_0(\Dom)$ such that, for all $K \in \calT_h$,
\begin{equation}
\label{eq:stab_Sch}
\begin{aligned}
h_K^{-1} \|q\|_K + \|\GRAD q\|_K
&\leq
C(\kappa_{\calT_h}) \|\bw\|_{\tK},
\\
h_K^{-1} \|\bphi\|_K + \|\GRAD \bphi\|_K
&\leq
C(\kappa_{\calT_h}) \|\ROTZ \bw\|_{\tK}.
\end{aligned}
\end{equation}

We now pick an arbitrary test function $\bv \in \Hrotz$. We have
\begin{equation*}
\bs(\be,\bv) = \bs(\be,\bv-\bS_{h0}\upc(\bv)) = \bs(\be,\GRAD q+\bphi),
\end{equation*}
where $q$ and $\bphi$ are the components of the decomposition in \eqref{tmp_decomposition_schoberl}.
We then estimate separately the two parts of the residual associated with the decomposition.

For the gradient part, we write
\begin{align*}
\bs(\be,\GRAD q)
&=
\bs(\be,\GRAD (q-\Ig (q)))
\\
&=
-\omega^2 (\eps\be,\GRAD (q-\Ig (q)))
\\
&=
\sum_{K \in \calT_h}
\omega^2 (\DIV (\eps\be),q-\Ig (q))_K
-
\sum_{F \in \calF_h^{\rm int}}
\omega^2 (\jump{\eps(\be)}_F\upd,q-\Ig (q))_F
\\
&=
\sum_{K \in \calT_h}
- (\DIV (\bJ + \omega^2 \eps \bE_h),q-\Ig (q))_K
+
\sum_{F \in \calF_h^{\rm int}}
\omega^2 (\jump{\eps\bE_h}_F\upd,q-\Ig (q))_F
\\
&\leq
\sum_{K \in \calT_h}
\left \{
\|\DIV (\bJ + \omega^2 \eps \bE_h)\|_K \|q-\Ig (q)\|_K
+
\omega^2 \|\jump{\eps\bE_h}_{\partial K}\upd\|_{\partial K \setminus \partial \Omega}
\|q-\Ig (q)\|_{\partial K\setminus \partial \Omega}
\right \}
\\
&\leq
\sum_{K \in \calT_h}
\etadK
\varepsilon_{\min,\tttK}^{\frac12}
\omega \left \{
\frac{k}{h_K} \|q-\Ig (q)\|_K
+
\bigg( \frac{k^2}{h_K}\bigg)^{\frac12} \|q-\Ig (q)\|_{\partial K}
\right \}.
\end{align*}
For all $K \in \calT_h$, invoking~\eqref{eq:inter_Ig} and \eqref{eq:stab_Sch}, we have
\begin{align*}
\frac{k}{h_K} \|q-\Ig (q)\|_K
+
\sqrt{\frac{k}{h_K}} \|q-\Ig (q)\|_{\partial K}
&\leq
C(\kappa_{\calT_h}) \|\nabla q\|_{\ttK} \\
&\leq
C(\kappa_{\calT_h}) \|\bv\|_{\tttK}
\leq
C(\kappa_{\calT_h})\varepsilon_{\min,\tttK}^{-\frac12} \|\bv\|_{\eps,\tttK}.
\end{align*}
Summing over $K \in \calT_h$ and since the number of overlaps is uniformly controlled by $\kappa_{\calT_h}$, we obtain
\begin{equation} \label{eq:bnd_apost_1}
|\bs(\be,\GRAD q)|
\leq
C(\kappa_{\calT_h}) \etad \omega \|\bv\|_\eps.
\end{equation}

For the $\bH^1_0(\Dom)$-part, proceeding similarly gives
\begin{align*}
\bs(\be,\bphi)
={}&
\bs(\be,\bphi-\bIc(\bphi))
\\
={}&
(\bJ+\omega^2\eps\bE_h,\bphi-\bIc(\bphi)) - (\bChz(\bE_h),\ROTZ(\bphi-\bIc(\bphi)))_{\bnu}
\\
={}&
\sum_{K \in \calT_h}
(\bJ+\omega^2\eps\bE_h-\ROT(\bnu\bChz(\bE_h)),\bphi-\bIc(\bphi))_K
\\ &-
\sum_{F \in \calF_h^{\rm int}} (\jump{\bnu\bChz(\bE_h)}_F\upc,\bphi-\bIc(\bphi))_F
\\
\leq{}&
\sum_{K \in \calT_h}
\etacK \nu_{\min,\tttK}^{\frac12}
\left \{
\frac{k}{h_K}\|\bphi-\bIc(\bphi)\|_K
+ \bigg( \frac{k^2}{h_K}\bigg)^{\frac12}
\|(\bphi-\bIc(\bphi)) {\times} \bn\|_{\partial K\setminus\partial\Omega}
\right \}
\\
\leq{}&
C(\kappa_{\calT_h})
\sum_{K \in \calT_h} \etacK \nu_{\min,\tttK}^{\frac12} \|\nabla \bphi\|_{\ttK}
\\
\leq{}&
C(\kappa_{\calT_h})
\sum_{K \in \calT_h} \etacK \nu_{\min,\tttK}^{\frac12} \|\ROTZ \bv\|_{\tttK}
\\
\leq{}&
C(\kappa_{\calT_h})
\sum_{K \in \calT_h} \etacK \|\ROTZ \bv\|_{\bnu,\tttK},
\end{align*}
so that
\begin{equation} \label{eq:bnd_apost_2}
|\bs(\be,\bphi)| \leq C(\kappa_{\calT_h}) \etac \|\ROTZ \bv\|_{\bnu}.
\end{equation}
Combining \eqref{eq:bnd_apost_1} and \eqref{eq:bnd_apost_2} concludes the proof.
\end{proof}

The next step is an Aubin--Nitsche-type duality argument
to estimate the $\bL_{\eps}^2$-norm of the error. Here, the weak consistency
estimate from Lemma~\ref{lem:weak_cons} is crucial to treat the nonconformity of the dG solution.

\begin{lemma}[$\bL_\eps^2$-norm reliability estimate] \label{lem:L2_residual}
We have
\begin{equation}
\label{eq_reliability_L2}
\omega \|\be\|_\eps \leq C(\kappa_{\calT_h})(1+\gp+\gd) \eta.
\end{equation}
\end{lemma}

\begin{proof}
Recall the $\bL^2_\eps$-orthogonal decomposition 
$\be = \btheta_0 + \btheta_\Pi$ with $\btheta_0 \in \Hrotzrotz^\perp$ and
$\btheta_\Pi \in \Hrotzrotz$ (see~\eqref{eq:def_e_theta}).

For the first component, recalling~\eqref{eq:theta_eps_sq} and using~\eqref{eq:GOd}, we have
\begin{equation}
\label{tmp_theta_zero_1}
\omega \|\btheta_0\|_\eps^2
=
\omega (\be,\btheta_0)_\eps
=
\omega \bs(\be,\bzeta_{\btheta})
-
\omega \deltawkc(\bE_h,\bzeta_{\btheta}).
\end{equation}
Since $\bE_h$ satisfies the Galerkin orthogonality for conforming test functions and 
invoking the bound~\eqref{eq_residual_Hcurl} established 
in Lemma~\ref{lemma_pde_residual}, we have, for all $\bv_h\upc \in \Pkbrotz$,
\begin{equation}
\label{tmp_theta_zero_2}
\omega \bs(\be,\bzeta_{\btheta})
=
\omega \bs(\be,\bzeta_{\btheta}-\bv_h\upc)
\leq
C(\kappa_{\calT_h}) \etadc \omega \tnorm{\bzeta_{\btheta}-\bv_h\upc}
\leq
C(\kappa_{\calT_h}) \gp \etadc\|\btheta_0\|_\eps,
\end{equation}
where we used the definition \eqref{eq:def_gp} of $\gp$ in the last inequality (since $\bv_h\upc$ is arbitrary in $\Pkbrotz$).
Moreover, owing to the estimate \eqref{eq:adj_cons} from Lemma~\ref{lem:weak_cons},
we have
\begin{equation*}
\omega |\deltawkc(\bE_h,\bzeta_{\btheta})|
\leq
\snormnc{\bE_h} \gd \|\btheta_0\|_\eps.
\end{equation*}
Recalling the definition~\eqref{eq:def_JMP} of the $\snormnc{\SCAL}$-seminorm and using~\eqref{eq:approx_av_gcurl}, we infer that
\[
\snormnc{\bE_h} \le \tnormE{\bE_h-\calI_{h0}\upcav(\bE_h)} \le 
C(\kappa_{\calT_h}) \etaj.
\]
This gives
\begin{equation} \label{tmp_theta_zero_3}
\omega |\deltawkc(\bE_h,\bzeta_{\btheta})|
\leq C(\kappa_{\calT_h}) \etaj \gd \|\btheta_0\|_\eps.
\end{equation}
Combining \eqref{tmp_theta_zero_1}, \eqref{tmp_theta_zero_2} and \eqref{tmp_theta_zero_3},
we arrive at
\begin{equation}
\label{tmp_apost_L2_theta_zero}
\omega \|\btheta_0\|_\eps
\leq
C(\kappa_{\calT_h}) (\gp+\gd) \eta.
\end{equation}

For the other part of the error, since $\btheta_\Pi \in \Hrotzrotz$,
we can use \eqref{eq_residual_Hcurl} to write
\begin{equation}
\label{tmp_apost_L2_theta_pi}
\omega^2 \|\btheta_\Pi\|_\eps^2
=
-\bs(\be,\btheta_\Pi)
\leq
C(\kappa_{\calT_h}) \etadc
\tnorm{\btheta_\Pi}
=
C(\kappa_{\calT_h}) \etadc \omega
\|\btheta_\Pi\|_\eps.
\end{equation}
Combining \eqref{tmp_apost_L2_theta_zero} and \eqref{tmp_apost_L2_theta_pi} proves \eqref{eq_reliability_L2}.
\end{proof}

We are now ready to establish a reliability estimate with an argument similar to
the one used in \cite{Chaumont:23_ipdghelmholtz} for the scalar Helmholtz problem.

\begin{theorem}[Reliability] \label{th:reliability}
We have
\begin{equation}
\label{eq_reliability}
\tnorm{\be}_{\dagger}
\leq
C(\kappa_{\calT_h})
\left (
1 + \max_{K \in \calT_h} \frac{\omega h_K}{k\vel_{\min,\chK}} + \gp+\gd
\right )\eta.
\end{equation}
\end{theorem}

\begin{proof}
Since $\tnorm{\be}_{\dagger}^2
= \tnorms{\be}^2 + \etaj^2$, we only need to estimate $\tnorms{\be}^2$. 
To this purpose, recall that the bilinear form $\bs^+$ defined in~\eqref{eq:def_bs_+}
is the inner product associated with the $\tnorms{{\cdot}}$-norm. We introduce the 
$\Hrotz$-conforming projection $\Bestc:\bVsh\to \Hrotz$ such that 
$\bs^+(\bv-\Bestc(\bv),\bw)=$ for all $\bv\in \bVsh$ and all $\bw\in\Hrotz$.
Reasoning as in the proof of Lemma~\ref{lem:bestb} proves 
the following Pythagorean identity:
\begin{equation} \label{eq:Pyth_error}
\tnorms{\be}^2 = \tnorms{\bE-\Bestc(\bE_h)}^2 + \tnorms{\bE_h-\Bestc(\bE_h)}^2.
\end{equation}
We estimate separately the two terms on the right-hand side.

For the first term, we observe that $\Bestc(\bE)=\bE$ and $\bE-\Bestc(\bE_h) \in \Hrotz$, and write that
\begin{align*}
\tnorms{\bE-\Bestc(\bE_h)}^2
&=
\bs^+(\bE-\Bestc(\bE_h),\bE-\Bestc(\bE_h)) \\
&= \bs^+(\Bestc(\be),\bE-\Bestc(\bE_h)) \\
&=
\bs^+(\be,\bE-\Bestc(\bE_h))
\\
&=
\bs(\be,\bE-\Bestc(\bE_h))+2\omega^2(\be,\bE-\Bestc(\bE_h))_\eps.
\end{align*}
Since the second argument in the first term on the right-hand side
is conforming, this term can be estimated by means of the estimate
\eqref{eq_residual_Hcurl} from Lemma~\ref{lemma_pde_residual}.
This gives
\begin{equation*}
|\bs(\be,\bE-\Bestc(\bE_h))|
\leq
C(\kappa_{\calT_h})\etadc \tnorm{\bE-\Bestc(\bE_h)}.
\end{equation*}
We apply the Cauchy--Schwarz inequality to bound the second term, leading to
\begin{equation*}
\omega^2|(\be,\bE-\Bestc(\bE_h))_\eps|
\leq
\omega\|\be\|_\eps\omega \|\bE-\Bestc(\bE_h)\|_\eps
\leq
\omega\|\be\|_\eps \tnorm{\bE-\Bestc(\bE_h)}.
\end{equation*}
This yields the following estimate:
\begin{equation*}
\tnorms{\bE-\Bestc(\bE_h)}
\leq
C(\kappa_{\calT_h}) \etadc + \omega \|\be\|_\eps
\leq
C(\kappa_{\calT_h})(1+\gp+\gd) \eta,
\end{equation*}
where we employed the $\bL_\eps^2$-estimate \eqref{eq_reliability_L2} from Lemma~\ref{lem:L2_residual}.

For the second term on the right-hand side of~\eqref{eq:Pyth_error}, 
invoking~\eqref{eq:approx_av_sharp} gives
\[
\tnorms{\bE_h-\Bestc(\bE_h)} \leq 
\tnorms{\bE_h - \calI_{h0}\upcav(\bE_h)} \leq 
C(\kappa_{\calT_h}) \left (
1 + \max_{K \in \calT_h} \frac{\omega h_K}{k\vel_{\min,\chK}}
\right )\etaj.
\]
Putting everything together yields the assertion.
\end{proof}

\subsection{Local error lower bound (local efficiency)}

We now derive efficiency estimates. To do so, we will need bubble functions
(see \cite{melenk_wohlmuth_2001a} and \cite[Section 2.7]{chaumontfrelet_vega_2022a}).
Specifically, for all $K \in \calT_h$, there exists a function $b_K \in H^1_0(K)$
with $b_K \leq 1$ such that, for all $\bw_K\in \bpolP_{k,d}$ (recall that $k\ge1$ by assumption),
\begin{equation}
\label{eq:bubble}
\|\bw_K\|_K \leq C(\kappa_{\calT_h}) k\|b_K^{\frac12}\bw_K\|_K,
\end{equation}
and
\begin{equation}
\label{eq:inverse}
\|\GRAD(b_K\bw_K)\|_K \leq C(\kappa_{\calT_h}) \frac{k}{h_K} \|b_K^{\frac12} \bw_K\|_K.
\end{equation}
Similarly, for all $F \in \calF_h$, there exists a function $b_F \in H^1_0(F)$
such that, for all $\bw_F \in \bpolP_{k,d-1}$,
\begin{equation}
\label{eq:bubble_face}
\|\bw_F\|_F \leq C(\kappa_{\calT_h}) k\|b_F^{\frac12}\bw_F\|_F,
\end{equation}
and an extension operator $\calE_{F}: \bpolP_{k,d-1} \to \bH^1_0(\tF)$ such that,
for all $\bw_F \in \bpolP_{k,d-1}$, $\calE_F(b_F\bw_F)|_F = b_F\bw_F$ and
\begin{equation}
\label{eq:stab_extension}
k h_F^{-\frac12} \|\calE_F(b_F\bw_{F})\|_{\tF}
+
k^{-1} h_F^{\frac12} \|\GRAD \calE_F(b_F\bw_{F})\|_{\tF}
\leq C(\kappa_{\calT_h})\|b_F^{\frac12}\bw_{F}\|_F.
\end{equation}

\begin{theorem}[Local efficiency] \label{th:efficiency}
For all $K \in \calT_h$, we have
\begin{equation}
\eta_K
\leq
C(\kappa_{\calT_h}) \calK_K^{\frac12} k^{\frac32}\left \{ 
\left (
1 + \frac{\omega h_K}{k \vel_{\min,\Kupf}}
\right )
\tnorm{\bE-\bE_h}_{\dagger,\Kupf}
+
\operatorname{osc}_{\Kupf}
\right \},
\end{equation}
with the data oscillation term
\begin{multline}
\operatorname{osc}_{\Kupf}^2
:=
\frac{1}{\omega^2}
\varepsilon_{\min,\Kupf}^{-1}
\times
\\
\sum_{K' \in \Kupf}
\min_{\bJ_{h} \in \Pkb}
\left \{
\frac{\omega^2 h_{K'}^2}{k^2 \vel_{\min,\Kupf}^2} \|\bJ-\bJ_{h}\|_{K'}^2
+
\frac{h_{K'}^2}{k^2} \|\DIV (\bJ-\bJ_{h})\|_{K'}^2
\right \},
\end{multline}
and the contrast coefficient
$\calK_K
:=
\max
\left \{
\frac{\varepsilon_{\max,\Kupf}}{\varepsilon_{\min,\tttK}},
\frac{\nu_{\max,\Kupf}}{\nu_{\min,\tttK}}
\right \}.$
\end{theorem}

\begin{proof}
Fix $K\in\calT_h$. The proof contains three parts,
respectively dedicated to providing upper bounds
for $\etadK$, $\etacK$ and $\etajK$.
\\ [10px]
\noindent {\bf (i)}
The proof that
$\etadK
\leq
C(\kappa_{\calT_h}) k^{\frac32}
\big (
\omega \|\bE-\bE_h\|_{\eps,\tK}
+
\operatorname{osc}_{\tK}
\big )
$
can be found in \cite[Lemma 3.5]{chaumontfrelet_vega_2022a}.
This proof is established for conforming edge finite elements,
but it holds verbatim in the discontinuous Galerkin setting.
\\ [10px]
\noindent {\bf (ii)}
For $\etacK$, we need to slightly adapt the proof from
\cite[Lemma 3.6]{chaumontfrelet_vega_2022a}. 
The volumic residual and jump term in
$\etacK$ are estimated separately.
\\ [5px]
\noindent {\bf (iia)}
For the volume term, we introduce $\br_K := \bJ_K+\omega^2\eps\bE_h-\ROT(\bnu\bChz(\bE_h))|_K$
and $\bv_K = b_K \br_K$, with $\bJ_K$ arbitrary in $\bpolP_{k,d}$. 
We observe that $\bv_K$ vanishes on $\partial K$, so that, letting $\bv_h$
be the zero-extension of $\bv_K$ to $\Dom$, we infer that
\begin{equation*}
\|b_K^{\frac12}\br_K\|_K^2
=
(\br_K,\bv_K)_K
=
(\bJ_K,\bv_K)_K-\bss(\bE_h,\bv_h)
=
\bss(\bE-\bE_h,\bv_h)-(\bJ-\bJ_K,\bv_K)_K.
\end{equation*}
For the first term, we employ $b_K\le 1$ and \eqref{eq:inverse} to show that
\begin{align*}
&|\bss(\bE-\bE_h,\bv_h)|
\\
&\leq
\omega\|\bE-\bE_h\|_{\eps,K}\omega\epsilon_{\max,K}^{\frac12}\|\bv_K\|_K
+
\|\bChz(\bE-\bE_h)\|_{\bnu,K}\nu_{\max,K}^{\frac12}\|\ROT\bv_K\|_K
\\
&\leq
C(\kappa_{\calT_h})
\left (
(\omega\epsilon_{\max,K}^{\frac12})
\omega\|\bE-\bE_h\|_{\eps,K}
+
\nu_{\max,K}^{\frac12}
\frac{k}{h_K}\|\bChz(\bE-\bE_h)\|_{\bnu,K}
\right )
\|b_K^{\frac12}\br_K\|_K
\\
&\leq
C(\kappa_{\calT_h})
\nu_{\max,K}^{\frac12}
\frac{k}{h_K}
\left (
\frac{\omega h_K}{k\vel_K}
\omega\|\bE-\bE_h\|_{\eps,K}
+
\|\bChz(\bE-\bE_h)\|_{\bnu,K}
\right )
\|b_K^{\frac12}\br_K\|_K.
\end{align*}
The second term is simply estimated as follows
\begin{align*}
|(\bJ-\bJ_K,\bv_K)_K|
\leq
\|\bJ-\bJ_K\|_K\|\bv_K\|_K
\leq
C(\kappa_{\calT_h})
\left (\nu_{\max,K}^{\frac12}\frac{k}{h_K}\right )
\nu_{\max,K}^{-{\frac12}}
\frac{h_K}{k}
\|\bJ-\bJ_K\|_K\|b_K^{\frac12}\br_K\|_K.
\end{align*}
Combining these two bounds leads to
\begin{equation*}
\nu_{\max,K}^{-{\frac12}}\frac{h_K}{k}
\|b_K^{\frac12}\br_K\|_K
\leq
C(\kappa_{\calT_h})
\left \{
\left (1 + \frac{\omega h_K}{k\vel_K}\right )
\tnorm{\bE-\bE_h}_{\sharp,K}
+
\nu_{\max,K}^{-{\frac12}} \frac{h_K}{k} \|\bJ-\bJ_K\|_K
\right \},
\end{equation*}
and therefore
\begin{align*}
\nu_{\max,K}^{-{\frac12}}\frac{h_K}{k}
\|\br_K\|_K
&\leq C(\kappa_{\calT_h})
k\nu_{\max,K}^{-{\frac12}}\frac{h_K}{k}
\|b_K^{\frac12}\br_K\|_K
\\
&\leq
C(\kappa_{\calT_h})
k
\left \{
\left (1 + \frac{\omega h_K}{k\vel_K}\right )
\tnorm{\bE-\bE_h}_{\sharp,K}
+
\nu_{\max,K}^{-{\frac12}} \frac{h_K}{k} \|\bJ-\bJ_K\|_K
\right \}.
\end{align*}
Invoking the triangle inequality leads to
\begin{multline}
\label{tmp_volume_residual}
\nu_{\max,K}^{-\frac12}\frac{h_K}{k}
\|\bJ+\omega^2\eps\bE_h-\ROT(\bnu\bChz(\bE_h))\|_K \\
\leq
C(\kappa_{\calT_h})
k
\left \{
\left (1 + \frac{\omega h_K}{k\vel_K}\right )
\tnorm{\bE-\bE_h}_{\sharp,K}
+
\nu_{\max,K}^{-\frac12} \frac{h_K}{k} \|\bJ-\bJ_K\|_K
\right \}.
\end{multline}
\\ [5px]
\noindent {\bf (iib)}
For the jump term, we introduce $\br_F := \jump{\bnu\bChz(\bE_h)}_F\upc$
and $\bv_F := \calE_F(b_F\br_F) \in \bH^1_0(\tF)$. Since $\bv_F \in \bH^1_0(\tF)$
and $\bv_F|_F = b_F\br_F$, we have
\begin{align}
\label{tmp_jump}
\|b_F^{\frac12}\br_F\|_F^2
&=
(\br_F,\bv_F)_F
\\
\nonumber
&=
(\ROTh (\bnu\bChz(\bE_h)),\bv_F)_{\tF}
-
(\bnu\bChz(\bE_h),\ROTh \bv_F)_{\tF}
\\
\nonumber
&=
(\bnu\bChz(\bE-\bE_h),\ROT \bv_F)_{\tF}
-
(\ROTh (\bnu\bChz(\bE-\bE_h)),\bv_F)_{\tF}.
\end{align}
For the first term, we can immediately write that
\begin{align*}
|(\bnu\bChz(\bE-\bE_h),\ROT \bv_F)_{\tF}|
&\leq
\nu_{\max,\tF}^{\frac12}\|\bChz(\bE-\bE_h)\|_{\bnu,\tF}\|\ROT \bv_F\|_{\tF}
\\
&\leq
C(\kappa_{\calT_h})
k h_F^{-\frac12} \nu_{\max,\tF}^{\frac12}\|\bChz(\bE-\bE_h)\|_{\bnu,\tF}\|b_F^{\frac12}\br_F\|_{F},
\end{align*}
where we employed \eqref{eq:stab_extension}. We infer that
\begin{multline}
\label{tmp_jump_1}
\nu_{\max,\tF}^{-\frac12}
\bigg( \frac{h_F}{k} \bigg)^{\frac12}
|(\bnu\bChz(\bE-\bE_h),\ROT \bv_F)_{\tF}|
\leq
C(\kappa_{\calT_h})
k^{\frac12}\|\bChz(\bE-\bE_h)\|_{\bnu,\tF}\|b_F^{\frac12}\br_F\|_{F}.
\end{multline}
For the second term, we first observe that
\begin{align*}
\ROTh (\bnu\bChz(\bE-\bE_h))
&=
\bJ+\omega^2\eps\bE - \ROTh(\bnu\bChz(\bE_h))
\\
&=
\omega^2\eps(\bE-\bE_h) + \bJ + \omega^2\eps\bE_h - \ROTh(\bnu\bChz(\bE_h)),
\end{align*}
and therefore,
\begin{align*}
&\|\ROTh (\bnu\bChz(\bE-\bE_h)\|_{\tF}
\\
&\leq
\omega
\epsilon_{\max,\tF}^{\frac12}
\omega\|\bE-\bE_h\|_{\eps,\tF}
+
\|\bJ + \omega^2\eps\bE_h - \ROTh(\bnu\bChz(\bE_h))\|_{\tF}
\\
&\leq \nu_{\max,\tF}^{\frac12} \frac{k}{h_{F}}
\left (
\frac{\omega h_{F}}{k\vel_{\tF}}
\omega\|\bE-\bE_h\|_{\eps,\tF}
+
\nu_{\max,\tF}^{-\frac12}\frac{h_{F}}{k}\|\bJ + \omega^2\eps\bE_h - \ROTh(\bnu\bChz(\bE_h))\|_{\tF}
\right ).
\end{align*}
Combining this last estimate with \eqref{eq:stab_extension}
and invoking the Cauchy--Schwarz inequality gives
\begin{multline}
\label{tmp_jump_2}
\nu_{\max,\tF}^{-\frac12} \bigg( \frac{h_F}{k} \bigg)^{\frac12}
|(\ROT (\bnu\bChz(\bE-\bE_h)),\bv_F)_{\tF}|
\leq
C(\kappa_{\calT_h}) \times
\\
\left (
\frac{\omega h_{F}}{k \vel_{\tF}}
\omega\|\bE-\bE_h\|_{\eps,\tF}
+
\nu_{\max,\tF}^{-\frac12}\frac{h_{F}}{k}\|\bJ + \omega^2\eps\bE_h - \ROTh(\bnu\bChz(\bE_h))\|_{\tF}
\right )
k^{\frac12}\|b_F^{\frac12}\br_F\|_F.
\end{multline}
We can now plug \eqref{tmp_jump_1} and \eqref{tmp_jump_2} in \eqref{tmp_jump},
leading to
\begin{align*}
\nu_{\max,\tF}^{-\frac12}&\bigg( \frac{h_F}{k} \bigg)^{\frac12}
\|b_F^{\frac12} \br_F\|_F
\\
\leq{}&
C(\kappa_{\calT_h})
k^{\frac12}
\left (
\frac{\omega h_{F}}{k\vel_{\tF}}
\omega\|\bE-\bE_h\|_{\eps,\tF}
+
\nu_{\max,\tF}^{-\frac12}\frac{h_{F}}{k}\|\bJ + \omega^2\eps\bE_h - \ROTh(\bnu\bChz(\bE_h))\|_{\tF}
\right )
\\
&+
C(\kappa_{\calT_h})
k^{\frac12}\|\bChz(\bE-\bE_h)\|_{\bnu,\tF}
\\
\leq{}&
C(\kappa_{\calT_h}) k^{\frac12} \left (1+ \frac{\omega h_{F}}{k\vel_{\tF}}\right )
\tnorm{\bE-\bE_h}_{\sharp,\tF}
\\
&+
C(\kappa_{\calT_h}) k^{\frac12}
\nu_{\max,\tF}^{-\frac12}\frac{h_{F}}{k}\|\bJ + \omega^2\eps\bE_h - \ROTh(\bnu\bChz(\bE_h))\|_{\tF}
\\
\leq{}&
C(\kappa_{\calT_h}) k^{\frac12}
\left (
\left (1+ \frac{\omega h_F}{k\vel_{\tF}}\right )
\tnorm{\bE-\bE_h}_{\sharp,\tF}
+
\nu_{\max,\tF}^{-\frac12}\frac{h_F}{k}\|\bJ-\bJ_h\|_{\tF}
\right ),
\end{align*}
owing to \eqref{tmp_volume_residual} and the shape-regularity of the mesh. Recalling
the definition of $\Kupf$,
it follows from \eqref{eq:bubble_face} that
\begin{equation*}
\etacK
\leq
C(\kappa_{\calT_h}) k^{\frac32}
\left (
\left (1+ \frac{\omega h_K}{k\vel_{\Kupf}}\right )
\tnorm{\bE-\bE_h}_{\sharp,\chK}
+
\nu_{\max,\Kupf}^{-\frac12}\frac{h_K}{k}\|\bJ-\bJ_h\|_{\Kupf} \right ).
\end{equation*}
\\ [10px]
\noindent {\bf (iii)}
For the last part of the estimator, we simply use that
$\etajK^2
\leq
\tnorm{\bE-\bE_h}_{\sharp,K}^2+\etajK^2
=
\tnorm{\bE-\bE_h}_{\dagger,K}^2.$
\end{proof}

\section{Bound on approximation and divergence conformity factors}
\label{sec:bnd_app_fac}

In this section, we show that the factors introduced in Section~\ref{sec:def_fac}
tend to zero as the mesh is refined.
For positive real numbers $A$ and $B$, we abbreviate as $A\lesssim B$
the inequality $A\le CB$ with a generic (nondimensional) constant $C$ whose value
can change at each occurrence as long as it is independent of the mesh size, the
frequency parameter $\omega$, and, whenever relevant, any function involved in
the bound. The constant $C$ can depend on the shape-regularity of the mesh,
the polynomial degree $k$, the (global) contrast in the coefficients
(i.e. $\epsilon_{\max}/\epsilon_{\min}$ and $\nu_{\max}/\nu_{\min}$),
and the shape of the domain $\Dom$ (but not on its size).

For simplicity, we focus on the case where the parameters $\eps$ and $\bnu$
are piecewise constant on a polyhedral partition of $\Dom$\rev{, and refer
the reader to Remark~\ref{rem:rough} for the more general case where the 
material coefficients are just bounded from above and from below away from zero.}
Under the assumption of piecewise constant coefficients
(see \cite{costabel_dauge_nicaise_1999a,Jochmann_maxwell_1999,BoGuL:13}), 
there exists $s\in (0,\frac12)$
such that, for all $\bv \in \Hrotz$ with $\eps\bv \in \Hdivdivz$, and for all $\bw \in \Hrot$ with
$\bnu^{-1} \bw \in \Hdivzdivz$, we have $\bv,\bw \in \bH^s(\Dom)$ with
\begin{equation}
|\bv|_{\bH^s(\Dom)}
\lesssim
\ell_\Dom^{1-s} \nu_{\min}^{-\frac12} \|\ROTZ \bv\|_{\bnu},
\qquad
|\bw|_{\bH^s(\Dom)}
\lesssim
\ell_\Dom^{1-s} \epsilon_{\max}^{\frac12} \|\ROT \bw\|_{\eps^{-1}}.
\end{equation}
The length scale $\ell_D$, e.g., the diameter of $D$, is introduced
for dimensional consistency. We will also need commuting (quasi-)interpolation operators,
$\calJ_{h}\upc: \Hrot \to \Pkbrot$ and $\calJ_{h}\upd: \Hdiv
\to \bP\upd_{k}(\calT_h) := \Pkb \cap \Hdiv$
such that $\ROT (\calJ_{h}\upc(\bv)) = \calJ_{h}\upd(\ROT\bv)$ for all
$\bv \in \Hrot$ and
\begin{equation}
\label{eq_interpolation}
\|\bv-\calJ_{h}\upc(\bv)\| \lesssim h^s |\bv|_{\bH^s(\Dom)},
\qquad
\|\bw-\calJ_{h}\upd(\bw)\| \lesssim \|\bw\|,
\end{equation}
for all $\bv \in \Hrot \cap \bH^s(\Dom)$ and all $\bw \in \Hdiv$.
Since we are working on simplicial meshes, we can invoke the
operators devised in \cite[Section 20]{EG_volI} (see also
\cite{Schoberl:01,ArnFW:06,Christiansen:07,ChrWi:08}) using edge (N\'ed\'elec) and
Raviart--Thomas finite elements. 

\begin{proposition}[Primal approximation factor]
Let $\gp$ be defined in \eqref{eq:def_gp}. We have
\begin{equation*}
\gp 
\lesssim
(1+\bst)
\left (
\frac{\omega\ell_{\Dom}}{\vel_{\min}}
\right )^{1-s}
\left (
\frac{\omega h}{\vel_{\min}}
\right )^{s},
\end{equation*}
where $\bst$ is the stability constant introduced in \eqref{eq:def_bst}.
\end{proposition}

\begin{proof}
See \cite[Lemma 5.1]{ThCFE:23}.
\end{proof}

\begin{lemma}[Dual approximation factor]
Let $\gd$ be defined in \eqref{eq:def_gd}. We have
\begin{equation}
\label{eq_bound_gd}
\gd
\lesssim
(1+\bst)
\left (\frac{\omega \ell_{\Dom}}{\vel_{\min}}\right )^{1-s}
\left (\frac{\omega h}{\vel_{\min}}\right )^s.
\end{equation}
\end{lemma}

\begin{proof}
Consider a right-hand side $\btheta \in \Hrotzrotz^\perp$ and the associated adjoint solution
$\bzeta_{\btheta} \in \Hrotz$ defined in \eqref{eq:adjoint}. Set
$\bphi_{\btheta} := \bnu \ROTZ \bzeta_{\btheta}$. The strong form of Maxwell's equations
ensures that
\begin{equation*}
\ROT \bphi_{\btheta} = \eps \btheta+\omega^2 \eps \bzeta_\btheta,
\end{equation*}
so that $\bphi_{\btheta} \in \Hrot$ with
\begin{equation}
\label{tmp_rot_phi}
\|\ROT \bphi_{\btheta}\|_{\eps^{-1}}
\leq
\|\btheta\|_{\eps}+\omega^2 \|\bzeta_\btheta\|_{\eps}
\leq
(1+\bst) \|\btheta\|_{\eps}.
\end{equation}
Besides, since $\bnu^{-1}\bphi_{\btheta} \in \Hdivzdivz$,
we infer that $\bphi_{\btheta} \in \bH^s(\Dom)$ with
\begin{equation}
\label{tmp_phi_Hs}
|\bphi_{\btheta}|_{\bH^s(\Dom)}
\lesssim
\ell_{\Dom}^{1-s}\epsilon_{\max}^{\frac12}
\|\ROT \bphi_{\btheta}\|_{\eps^{-1}}.
\end{equation}

We are now ready to bound $\gd$. We notice that
\begin{align*}
\gd^2 &\le \omega^2 \tnormH{\bphi_\btheta-\calJ_{h}\upc (\bphi_\btheta)}^2 \\
&= \omega^2 \|\bphi_\btheta-\calJ_{h}\upc (\bphi_\btheta)\|_{\bnu^{-1}}^2
+ \omega^2 \|\tnu^{-\frac12}\th(\ROT\bphi_\btheta-\calJ_{h}\upd (\ROT \bphi_\btheta))\|^2,
\end{align*}
where we employed the commuting property satisfied by $\calJ_{h}\upc$ and $\calJ_{h}\upd$.
We bound the two terms on the right-hand side. 
For the first term, invoking \eqref{eq_interpolation} and \eqref{tmp_phi_Hs}, we have
\begin{multline}
\label{tmp_gamma_L2}
\omega \|\bphi_\btheta-\calJ_{h}\upc (\bphi_\btheta)\|_{\bnu^{-1}}
\lesssim
\omega \nu_{\min}^{-\frac12} h^s |\bphi_\btheta|_{\bH^s(\Dom)}
\\
\lesssim
\omega \ell_{\Dom}^{1-s} \frac{h^s}{\vel_{\min}} \|\ROT \bphi_\btheta\|_{\eps^{-1}}
=
\left (\frac{\omega \ell_{\Dom}}{\vel_{\min}}\right )^{1-s}
\left (\frac{\omega h}{\vel_{\min}}\right )^s
\|\ROT \bphi_\btheta\|_{\eps^{-1}}.
\end{multline}
For the second term, using \eqref{eq_interpolation}, we can write
\begin{align}
\label{tmp_gamma_curl}
\omega \|\tnu^{-\frac12}\th(\ROT\bphi_\btheta-\calJ_{h}\upd (\ROT \bphi_\btheta))\|
&\leq
\omega h \nu_{\min}^{-\frac12}
\|\ROT\bphi_\btheta-\calJ_{h}\upd (\ROT \bphi_\btheta)\|
\nonumber \\
&\lesssim
\omega h \nu_{\min}^{-\frac12} \|\ROT\bphi_\btheta\|
\lesssim
\frac{\omega h}{\vel_{\min}}\|\ROT\bphi_\btheta\|_{\eps^{-1}}.
\end{align}
Combining \eqref{tmp_rot_phi}, \eqref{tmp_gamma_L2} and \eqref{tmp_gamma_curl}
and observing that $h\le \ell_D$ proves the assertion.
\end{proof}

\begin{lemma}[Divergence conformity factor]
Let $\gdiv$ be defined in~\eqref{eq_gamma_bX_DG}. 
We have
\begin{equation}
\label{eq:bnd_on_gdiv}
\gdiv
\lesssim
\left (\frac{\omega \ell_{\Dom}}{\vel_{\min}}\right )^{1-s}
\left (\frac{\omega h}{\vel_{\min}}\right )^s.
\end{equation}
\end{lemma}

\begin{proof}
Let $\bv_h \in \bX\upb_h$ and consider an arbitrary $\bv\upc_h\in \Pkbrotz$.
Since $\bPi\upc_{h0}(\bv_h) = \bzero$ by assumption, we have
\begin{align*}
\bPi\upc_0(\bv_h) &= \bPi\upc_0(\bv_h-\bv\upc_h) + \bPi\upc_0(\bv\upc_h) \\
&= \bPi\upc_0((I-\bPi\upc_{h0})(\bv_h-\bv\upc_h)) + \bPi\upc_0(\bv\upc_h-\bPi\upc_{h0}(\bv\upc_h)).
\end{align*}
Multiplying by $\omega$, and invoking the triangle inequality and the $\bL^2_\eps$-stability
of the projection operators, we infer that
\begin{equation} \label{eq:tmptmp}
\omega \|\bPi\upc_0(\bv_h)\|_\eps \leq \omega \|\bv_h-\bv\upc_h\|_\eps + \omega \|\bPi\upc_0(\bv\upc_h-\bPi\upc_{h0}(\bv\upc_h))\|_\eps.
\end{equation}
For the second term on the right-hand side of~\eqref{eq:tmptmp}, we invoke \cite[Lemma 5.2]{ThCFE:23} which gives
\begin{align*}
\omega \|\bPi\upc_0(\bv\upc_h-\bPi\upc_{h0}(\bv\upc_h))\|_\eps
&\lesssim
\left (\frac{\omega\ell_\Dom}{\vel_{\min}}\right )^{1-s}
\left (\frac{\omega h}{\vel_{\min}}\right )^{s}
\|\ROTZ(\bv\upc_h-\bPi\upc_{h0}(\bv_h\upc))\|_\bnu \\
&\lesssim
\left (\frac{\omega\ell_\Dom}{\vel_{\min}}\right )^{1-s}
\left (\frac{\omega h}{\vel_{\min}}\right )^{s}
\big(\|\bChz(\bv_h)\|_\bnu+ \|\bChz(\bv_h-\bv\upc_h)\|_\bnu\big),
\end{align*}
where we used that $\ROTZ\bPi\upc_{h0}(\bv_h\upc)=\bzero$ and the triangle inequality.
For the first term on the right-hand side of~\eqref{eq:tmptmp}, we observe that
\begin{equation*}
\omega \|\bv_h-\bv_h\upc\|_{\eps}
\leq
\frac{\omega h}{\vel_{\min}}
\|\tnu^{\frac12}\th^{-1}(\bv_h-\bv_h\upc)\|
\leq
\frac{\omega h}{\vel_{\min}}
\tnormE{\bv_h-\bv\upc_h}.
\end{equation*}
Combining this bound with the above two bounds and since $h\le \ell_D$, this gives
\begin{equation*}
\omega\|\bPi\upc_0(\bv_h)\|_\eps
\lesssim
\left (\frac{\omega\ell_\Dom}{\vel_{\min}}\right )^{1-s}
\left (\frac{\omega h}{\vel_{\min}}\right )^{s}
\big(
\|\bChz(\bv_h)\|_\bnu^2
+
\tnormE{\bv_h-\bv\upc_h}^2
\big)^{\frac12}.
\end{equation*}
The bound~\eqref{eq:bnd_on_gdiv} 
follows by taking the minimum over $\bv\upc_h \in \Pkbrotz$ and recalling the 
definition~\eqref{eq:def_JMP} of the $\snormnc{\SCAL}$-seminorm.
\end{proof}



\begin{remark}[Rough coefficients] \label{rem:rough}
It is possible to show that the above factors tend to zero \rev{(without a specific algebraic rate) when the material coefficients are just bounded from above and from below away from zero. The proof
hinges on a compactness result from \cite{Weber:80}.}
We refer the reader to the discussions in
\cite[Section 5.2]{ThCFE:23} \rev{which can be readily extended to the present
nonconforming setting. Details are skipped for brevity.}
\end{remark}

\begin{remark}[\rev{Improved decay rates}] \label{rem:improved}
One can show improved convergence rates for the approximation factors 
by assuming extra regularity \rev{on the material coefficients and the domain}. 
We refer the reader to the discussion in
\cite[Remark~5.5]{ThCFE:23} for more details.
\end{remark}

\ifHAL
\bibliographystyle{plain}
\bibliography{biblio}
\fi

\end{document}